\documentclass[bj,authoryear]{imsart-nobox}

\usepackage{amsmath}
\usepackage{amsfonts}
\usepackage{amssymb}
\usepackage{amsthm}
\usepackage{microtype,dsfont}
\pdfmapfile{+sansmathaccent.map} 

\startlocaldefs
\numberwithin{equation}{section}
\theoremstyle{plain}
\newtheorem{theorem}{Theorem}[section]
\newtheorem{lemma}[theorem]{Lemma}
\newtheorem{proposition}[theorem]{Proposition}

\theoremstyle{remark}

\newtheorem{definition}[theorem]{Definition}

\newtheorem*{notation}{Notation}

\renewcommand{\le}{\leqslant}

\renewcommand{\ge}{\geqslant}

\renewcommand{\emptyset}{\varnothing}

\newcommand{\toto}{\rightrightarrows}

\newcommand{\supp}{\operatorname{spt}} 
\newcommand{\spt}{\supp} 
\newcommand{\cl}{\operatorname{cl}} 
\newcommand{\intr}{\operatorname{int}} 
\newcommand{\dom}{\operatorname{dom}} 
\newcommand{\rge}{\operatorname{rge}} 
\newcommand{\proj}{\operatorname{proj}} 
\newcommand{\hzn}{\operatorname{hzn}} 
\newcommand{\ball}{\mathds{B}} 
\newcommand{\oball}{\ball^\circ}
\newcommand{\sphere}{\mathds{S}} 
\newcommand{\argmax}{\operatornamewithlimits{\arg\max}}

\newcommand{\restr}{\llcorner} 

\newcommand{\prob}{\operatorname{\mathbb{P}}}
\newcommand{\Fell}{\mathcal{F}}
\providecommand{\E}{\mathbb{E}}
\newcommand{\DD}{\mathbb{D}}

\newcommand{\reals}{\mathbb{R}}
\newcommand{\NN}{\mathbb{N}}
\newcommand{\id}{\mathrm{Id}}

\newcommand{\Rd}{\reals^{d}}
\newcommand{\Rdd}{\Rd \times \Rd}
\newcommand{\Rk}{\reals^{k}}
\newcommand{\Rl}{\reals^{\ell}}
\newcommand{\Rkl}{\Rk \times \Rl}

\newcommand{\eps}{\varepsilon}
\newcommand{\diff}{\mathrm{d}}

\newcommand{\Vto}{\raisebox{-0.5pt}{\,\scriptsize$\stackrel{\raisebox{-0.5pt}{\mbox{\tiny $V$}}}{\longrightarrow}$}\,}

\newcommand{\wto}{\raisebox{-0.5pt}{\,\scriptsize$\stackrel{\raisebox{-0.5pt}{\mbox{\tiny $\mathrm{w}$}}}{\longrightarrow}$}\,}

\newcommand{\inpr}[1]{\langle{#1}\rangle}
\newcommand{\norm}[1]{\left\lvert{#1}\right\rvert}

\newcommand{\con}{\operatorname{con}}
\newcommand{\Prob}{\mathcal{P}}

\providecommand{\Pcm}{\Prob_{\mathrm{cm}}}
\providecommand{\varPicm}{\varPi_{\mathrm{cm}}}

\providecommand{\Fellm}{\Fell_{\mathrm{m}}}
\providecommand{\Fellmm}{\Fell_{\mathrm{mm}}}
\providecommand{\Fellcm}{\Fell_{\mathrm{cm}}}
\providecommand{\Fellmcm}{\Fell_{\mathrm{mcm}}}

\endlocaldefs

\begin{document}

\begin{frontmatter}
\title{Graphical and uniform consistency of estimated optimal transport plans}
\runtitle{Estimated optimal transport plans}

\begin{aug}
\orcid{0000-0002-0444-689X}
\author[A]{\fnms{Johan}~\snm{Segers}\ead[label=e1]{johan.segers@uclouvain.be}}
\address[A]{UCLouvain, ISBA/LIDAM, Voie du Roman Pays 20, bte L1.04.01, 1348 Louvain-la-Neuve, Belgium\printead[presep={,\ }]{e1}}

\end{aug}

\begin{abstract}
A general theory is provided delivering convergence of maximal cyclically monotone mappings containing the supports of coupling measures of sequences of pairs of possibly random probability measures on Euclidean space. The theory is based on the identification of such a mapping with a closed subset of a Cartesian product of Euclidean spaces and leveraging tools from random set theory. Weak convergence in the appropriate Fell space together with the maximal cyclical monotonicity then automatically yields local uniform convergence of the associated mappings. Viewing such mappings as optimal transport plans between probability measures with respect to the squared Euclidean distance as cost function yields consistency results for notions of multivariate ranks and quantiles based on optimal transport, notably the empirical center-outward distribution and quantile functions.
\end{abstract}


\end{frontmatter}





\section{Introduction and main results}

\subsection{Motivation}
\label{subsec:motiv}

Let $\Prob(\Rd)$ be the space of Borel probability measures equipped with the weak topology. For $P, Q \in \Prob(\Rd)$, let $\varPi(P, Q)$ be the set of coupling measures $\pi \in \Prob(\Rdd)$ with margins $P$ and $Q$, i.e., $\pi(B \times \Rd) = P(B)$ and $\pi(\Rd \times B) = Q(B)$ for Borel sets $B \subset \Rd$. Recall that the support of a probability measure is the set of all points any open neighbourhood of which receives positive mass. Motivated by the theory of optimal transportation, we will say $\pi \in \varPi(P, Q)$ is optimal if its support, $\spt \pi$, is cyclically monotone (Definition~\ref{def:mon} below), which by Rockafellar's theorem means it is included in the subdifferential $T = \partial \psi \subset \Rdd$ of a closed convex function $\psi : \Rd \to \reals \cup \{+\infty\}$. Indeed, the Knott--Smith optimality criterion \citep[Theorem~2.12(i)]{Villani-Topics} states that, given $P$ and $Q$ in $\Prob(\Rd)$ with finite second moments, a coupling measure $\pi \in \varPi(P, Q)$ has cyclically monotone support if and only if it minimizes the expected transportation cost $\int \left|x-y\right|^2 \, \diff \pi'(x, y)$ with respect to the squared Euclidean distance over all couplings $\pi' \in \varPi(P, Q)$.

If $P$ vanishes on sets of Hausdorff dimension at most $d-1$, then the potential $\psi$ is differentiable $P$-almost everywhere and we have $\pi = (\id \times \nabla \psi)_\# \mu$ and $Q = (\nabla \psi)_\# P$, with $\id$ the identity map on $\Rd$, where $\nabla \psi$ is the gradient of $\psi$ and where $\#$ denotes the push forward of measures \citep{McCann}. If $P$ and $Q$ have finite second moments, then $\pi$ is the solution of the Kantorovich problem with respect to the squared Euclidean distance as cost function and $\nabla \psi$ is the Brenier map solving the associated Monge problem; see  \cite{Brenier} or \citet[Theorem~2.12(ii)]{Villani-Topics}. With some abuse of terminology, $T = \partial \psi$ will be called an optimal transport plan between $P$ and $Q$ even in the absence of second moments and when $\psi$ is not differentiable $P$-almost everywhere.

Suppose $P$ and/or $Q$ are unknown or are approximated by a sequence of possibly random probability measures $P_n$ and $Q_n$ converging in the weak topology to $P$ and $Q$, respectively. In case $P_n$ and $Q_n$ are random, convergence may take place almost surely or in distribution. The optimal transport plan $T$ may then be approximated or estimated by a sequence of (random) optimal transport plans $T_n = \partial \psi_n$ containing the supports of coupling measures $\pi_n \in \varPi(P_n, Q_n)$. The question is in which sense $T_n$ converges to $T$.

The problem received renewed attention in nonparametric statistics as optimal transport plans provide ways to define versatile notions of ranks and quantiles for multivariate distributions, opening the door to distribution-free hypothesis tests and semiparametrically efficient estimators. In all cases, the transport occurs between a target distribution, to be estimated from data, and a reference distributions, which may be discretized to facilitate computations.
A prominent example are the center-outward distribution and quantile functions in \citet{Chernozhukov}, \citet{delbarrio+g+h:2020} and \citet{hallin2021AoS}, with applications to quantile contours \citep{beirlant+b+db+h+k:2020}, distribution-free tests of independence \citep{shi2020JASA, shi2020rateoptimality} and multiple-output regression \citep{hallin2022JASA, delBarrio2022quantile}. Here, the reference distribution is spherically uniform on the unit ball.
Quantile and rank maps via semi-discrete optimal transport between a random sample from a target distribution to an absolutely continuous reference distribution are considered in \citet{ghosal2022multivariate}. In contrast, \citet{deb2019multivariate} propose to match sample points to a quasi-uniform grid on the unit cube. Matching sample points to fixed discrete grid arising from the discretisation of a general reference distribution also underlies the two-sample test procedures in \citet{deb2021efficiency}.
 
These articles convincingly argue the usefulness of the optimal transportation perspective. At the basis of the analysis are Glivenko--Cantelli type theorems showing the uniform consistency of the estimated transport plans. However, the conditions under which these are proved involve restrictions on the densities or on the supports of the measures $P$ and $Q$. The contribution of this paper is to provide such uniform consistency results under quite general conditions.
	
\subsection{Set convergence}

The perspective of this paper is to view the subdifferentials $T_n$ and $T$ as (random) closed subsets of $\Rdd$ and employ notions of set convergence. Uniform convergence of $T_n$ (as a multivalued mapping) to $T$, local or global, is then derived leveraging maximal (cyclical) monotonicity \citep{Rockafellar-Wets, A-A}. 

For a locally compact Hausdorff second countable (LCHS) space $\E$, let $\Fell(\E)$ denote the set of all closed subsets of $\E$. For $A \subset \E$, let $\Fell_A(\E) = \{ F \in \Fell(\E) : F \cap A \ne \emptyset \}$ and $\Fell^A(\E) = \{F \in \Fell(\E) : F \cap A = \emptyset \}$ denote the collections of closed sets that hit or miss $A$, respectively. The Fell hit-and-miss topology on $\Fell(\E)$ is the one generated by the collections $\Fell_G(\E)$ and $\Fell^K(\E)$, where $G$ and $K$ vary over the open and compact sets of $\E$, respectively. The Fell topology is metrizable and turns $\Fell(\E)$ into a compact metric space. In $\Fell(\E)$, we have $F_n \to F$ as $n \to \infty$ if and only if for every open $G \subset \E$ such that $F \cap G \neq \emptyset$ we have $F_n \cap G \neq \emptyset$ for all large $n$ and for every compact $K$ such that $F \cap K = \emptyset$ we have $F_n \cap K = \emptyset$ for all large $n$. The Fell topology is a classical subject in analysis, see for instance \cite{matheron1975}, \cite{beer1993}, and \cite{Molchanov2005}. Some useful facts are collected in the supplement (Appendix~\ref{app:Fell}). Its application to the theory of weak convergence of possibly discontinuous random processes was explored for instance in \cite{bucher+s+v:2014}.

For $T \in \Fell(\Rdd)$ and non-empty $V \subset \Rd$, let $T \restr V = T \cap (V \times \Rd)$ denote the restriction of $T$ to $V$. If $V$ is open, $V \times \Rd$ is open and thus LCHS, while $T \restr V$ belongs to $\Fell(V \times \Rd)$. In the context of the central question of this paper involving weakly converging probability measures $P_n \wto P$ and $Q_n \wto Q$, the key result will be the convergence of $T_n \restr V$ to $T \restr V$ in $\Fell(V \times \Rd)$ for $V = \intr(\spt P)$, where $T_n$ and $T$ are optimal transport plans between $P_n$ and $Q_n$ and between $P$ and $Q$, respectively. Thanks to uniqueness and boundedness properties of maximal (cyclically) monotone mappings, such set convergence will imply the Hausdorff and (locally) uniform convergence of the possibly multivalued mappings $x \mapsto T_n(x) = \{ y \in \Rd : (x, y) \in T_n \}$.

\subsection{Main results}

The stage is now almost ready for the main results of the paper. Let $\intr A$, $\cl A$, and $\partial A$ denote the interior, closure, and boundary, respectively, of a subset $A$ of a topological space. A set $T \subset \Rdd$ is identified with a multivalued mapping from $\Rd$ to $\Rd$ via $T(x) = \{ y \in \Rd : (x, y) \in T \}$. The domain and range of $T \subset \Rdd$ are $\dom T = \{ x \in \Rd : T(x) \ne \emptyset \}$ and $\rge T = \{ y \in \Rd : \exists x \in \Rd, y \in T(x) \}$. We say $T$ is single-valued in $x \in \Rd$ if $T(x)$ is a singleton, in which case $T(x)$ also denotes the single element $y$ such that $(x, y) \in T$. Let $|\cdot|$ denote the Euclidean norm and $\sphere_{d-1} = \{u \in \Rd : |u| = 1 \}$ the unit sphere. A set $E \subset \Rd$ recedes in direction $u \in \sphere_{d-1}$ if there exist $x_n \in E$ such that $|x_n| \to \infty$ and $x_n/|x_n| \to u$ as $n \to \infty$. A non-empty, compact, convex set $C \subset \Rd$ is strictly convex in direction $u \in \sphere_{d-1}$ if $\argmax_{c \in C} \inpr{c, u}$ is a singleton. Let $\wto$ denote weak convergence of Borel probability measures on a metric space and let $d_H$ denote the Hausdorff semi-metric: for $A, B \subset \Rd$, we have $d_H(A, B) = \max \{ \sup_{a \in A} d(a, B), \sup_{b \in B} d(b, A) \}$, where, for $x \in \Rd$ and non-empty $Y \subset \Rd$, we have $d(x, Y) = \inf_{y \in Y} |x-y|$, while $d(x, \emptyset) = \infty$. 

\begin{theorem}
\label{thm}
Let $P_n, P, Q_n, Q \in \Prob(\Rd)$ be such that $P_n \wto P$ and $Q_n \wto Q$ as $n \to \infty$. Let $\pi_n \in \varPi(P_n, Q_n)$ and $\pi \in \varPi(P, Q)$ be such that there exist maximal cyclically monotone $T_n, T \in \Fell(\Rdd)$ containing $\spt \pi_n$ and $\spt \pi$, respectively. 
Assume $\pi$ is the only coupling measure of $P$ and $Q$ with cyclically monotone support.
Write $V = \intr(\spt P)$ and assume $V$ is non-empty.

(a) We have $\pi_n \wto \pi$ in $\Prob(\Rdd)$ as well as $V \subset \dom T$ and
 $T_n \restr V \to T \restr V$ in $\Fell(V \times \Rd)$ as $n \to \infty$. In particular, for any compact $K \subset V \times \Rd$ such that $T \cap K = \emptyset$ we have $T_n \cap K = \emptyset$ for all large $n$, and for any open $G \subset V \times \Rd$ such that $T \cap G \ne \emptyset$ we have $T_n \cap G \ne \emptyset$ for all large $n$.

(b) For any compact $K \subset V$, we have $K \subset \intr(\dom T_n)$ and thus $T_n(K)$ is compact for large $n$.
If $T$ is single-valued on $\partial K$, then
\[
	d_H\bigl(T_n(K), T(K)\bigr) \to 0, \qquad n \to \infty.
\]
If $T$ is single-valued on the whole of $K$, then actually
\[
	\sup_{x \in K} \sup_{y \in T_n(x)} |y - T(x)| \to 0, \qquad n \to \infty.
\]
In particular, if $P(W) = 1$ for some Borel set $W \subset V$ on which $T$ is single-valued, then for $P$-almost every $x \in \Rd$, the set $T_n(x)$ is non-empty for all large $n$ and $\sup_{y \in T_n(x)} |y - T(x)| \to 0$ as $n \to \infty$.

(c) Assume that $\rge T$ is bounded, that there exists an open set $U \subset \spt P$ such that $T(U) \subset \intr(\rge T)$ and $\rge T \subset \cl(T(U))$, and that $\rge T_n \subset \cl(\rge T)$ for all but finitely many $n$. Then the domains of $T$ and $T_n$ are equal to $\Rd$ for all but finitely any $n$ and $T_n \to T$ in $\Fell(\Rdd)$ as $n \to \infty$.

(d) If, in addition to the conditions in (c), $T$ is single-valued on the closure of a set $E \subset \Rd$ that recedes only in directions in which the (necessarily convex) set $\cl(\rge T)$ is strictly convex, then
\[
	\sup_{x \in E} \sup_{y \in T_n(x)} |y - T(x)| \to 0, \qquad n \to \infty.
\]
\end{theorem}

If, for instance, $\cl(\rge T)$ in (d) is equal to a ball, then it is strictly convex in any direction, and so $E$ can potentially $\Rd$ itself. If, however, $E$ is the unit cube $[0, 1]^d$, then $E$ can recede only in directions $u = (u_1,\ldots,u_d) \in \sphere_{d-1}$ such that $u_j \ne 0$ for all $j \in \{1,\ldots,d\}$.

In comparison to Proposition~1.7.11 in \cite{panaretos+z:2020}, no finite second moments are needed in Theorem~\ref{thm} and the uniform convergence can take place potentially on unbounded sets. Part~(a) states that in any case, convergence takes place, albeit in a weaker topology than the one of uniform convergence, namely the Fell topology on a certain base space. Still, this may already provide some useful information about the asymptotic behaviour of $T_n$. Choosing for instance $G = V \times \{ y \in \Rd : |y| > \lambda \}$ in Theorem~\ref{thm}(a) for some $\lambda > 0$ implies that if $\sup_{x \in V} \sup_{y \in T(x)} |y| > \lambda$, then also $\sup_{x \in V} \sup_{y \in T_n(x)} > \lambda$ for all large $n$, and thus that $\liminf_{n \to \infty} \sup_{x \in V} \sup_{y \in T_n(x)} |y| \ge \sup_{x \in V} \sup_{y \in T(x)} |y|$. Lemma~\ref{lem:TmmKdeltaB} below, which is a step in the proof of Theorem~\ref{thm}, states that if $K$ is a compact subset of $V$, then for small $\delta > 0$ and large $n$, the sets $T_n(\{x \in \Rd : d(x, K) \le \delta\})$ and $T(K)$ are close in the Hausdorff metric, and this without any further conditions on $T$.

An interesting question is whether Theorem~\ref{thm} can be generalized to more general spaces and transportation costs. For the Fell topology to have convenient properties, the base space needs to be locally compact, second countable, and Hausdorff. Important in the proof are certain properties of maximal monotone functions in Euclidean space reviewed in \cite{A-A}. The proof of Theorem~\ref{thm:UCT}, which underlies part~(d) in Theorem~\ref{thm}, relies on the bilinearity of the inner product on $\Rd$. All in all, the extension to more general settings is a challenging topic for further research. 

In Theorem~\ref{thm}, the approximating measures $P_n$ and $Q_n$ are not yet random. 
The extension to random probability measures goes through as expected, although some measurability issues require special attention. Indeed, even if there is an event of probability one on which $P_n$ converges weakly to $P$ (for instance, by the law of large numbers in case of the empirical distribution), then it is not enough to conclude that on this event, the suprema in (b) and (d) of Theorem~\ref{thm} tend to zero too; indeed, Counterexample~1.9.4 in \cite{vdVW96} involves a sequence $X_n$ of real-valued maps on a probability space $\Omega$ such that $\lim_{n\to\infty} X_n(\omega) = 0$ for \emph{every} $\omega \in \Omega$ while still $X_n^* = 1$ for all $n$, with $X_n^*$ the measurable cover of $X_n$, so that $X_n$ does not converge to $0$ in outer probability. Recall that an $F_\sigma$-set is a countable intersection of closed sets. 

\begin{theorem}
\label{thm:random}
In Theorem~\ref{thm}, if instead $(P_n, Q_n, \pi_n, T_n)$ are random elements in $\Prob(\Rd) \times \Prob(\Rd) \times \Prob(\Rdd) \times \Fell(\Rdd)$ such that $\pi_n \in \varPi(P_n, Q_n)$ and $\spt \pi_n \subset T_n$ with probability one and if for some non-random $P, Q \in \Prob(\Rd)$ we have $P_n \wto P$ and $Q_n \wto Q$ in $\Prob(\Rd)$ almost surely or in distribution, then the stated convergence relations in (a)--(d) hold true almost surely or in distribution too, respectively. In (c), it is assumed that $\rge T_n \subset \cl(\rge T)$ with probability one, while in (d), it is assumed that $E$ is an $F_\sigma$-set, so that the double supremum is a random variable.
\end{theorem}

Coming back to the applications of optimal transport in statistics in Section~\ref{subsec:motiv}, Theorem~\ref{thm:random} applies to multivariate rank maps when $P_n$ is the empirical distribution of a random sample drawn from $P$ and $Q_n$ is a discretization of a reference distribution $Q$ or just $Q$ itself. Choosing $Q$ to be the spherically uniform distribution on the unit ball yields the center-outward distribution function \cite{hallin2021AoS}, whereas \cite{deb2019multivariate} opt for the uniform distribution on the unit cube. Switching the roles of $P$ and $Q$ yields notions of multivariate quantiles built from transporting a reference distribution $P$ (or a discretization thereof) to a random sample drawn from $Q$. In all these cases, Theorem~\ref{thm:random} provides general conditions under which the estimated maps converge at least graphically (i.e., as random sets in the Fell topology), almost surely or in distribution, and, given additional conditions, also uniformly, locally or sometimes even globally.


At the heart of the results lies the identification of a possibly multivalued mapping with a closed set. Graphical convergence of such mappings is then defined as convergence in the Fell topology on the space of closed subsets of the Cartesian product, a topic explored in Section~\ref{sec:mmm}. In case the mappings are maximal (cyclically) monotone and under side conditions, graphical convergence turns out imply to convergence in the Hausdorff metric and the supremum distance (Section~\ref{sec:convmon}). In the context of optimal transport, the mappings of interest contain the supports of couplings between two given probability measures. Some auxiliary results about these supports are given in Section~\ref{sec:coupling}. The proofs of Theorems~\ref{thm} and~\ref{thm:random} are provided in Section~\ref{sec:proofs}, followed by the proofs of the results in  Sections~\ref{sec:mmm}, \ref{sec:convmon} and~\ref{sec:coupling} in Appendices~\ref{app:mmm}, \ref{app:convmon} and~\ref{app:coupling}, respectively. The Supplement (Appendix~\ref{app:Fell}) provides a general exposition about the Fell topology.

\begin{notation}
The open and closed unit balls in $\Rd$ are $\oball = \{ x \in \Rd : |x| < 1 \}$ and $\ball = \{ x \in \Rd : |x| \le 1\}$, respectively. The Minkowski sum of $A, B \in \Rd$ is $A + B = \{ a + b : a \in A, b \in B\}$, while for $A \subset \Rd$ and $\lambda \in \reals$, we put $\lambda A = \{ \lambda a : a \in A \}$.
For $A \subset \Rd$, we thus have $A + \eps \ball = \{ a + x : a \in A, |x| \le \eps \}$ and $A + \eps \oball = \{ a + x : a \in A, |x| < \eps \}$. With slight abuse of notation, the open and closed balls of radius $\rho > 0$ centred at $x \in \Rd$ are $x + \rho \oball$ and $x + \rho \ball$, respectively.
\end{notation}

\section{Graphical convergence of multivalued mappings}
\label{sec:mmm}

To apply notions of set convergence to the study of sequences of functions between Euclidean spaces, we identify a subset $T$ of $\Rkl \equiv \reals^{k+\ell}$ with the multivalued mapping
\[ 
\Rk \toto \Rl : x \mapsto T(x) := \{ y \in \Rl : (x, y) \in T \}, 
\]
of which the set $T$ is the graph in the sense of \citet[p.~148]{Rockafellar-Wets}. The domain and the range of $T$ are $\dom T = \{ x \in \Rk : T(x) \ne \emptyset \}$ and $\rge T = \bigcup_{x \in \Rk} T(x)$ respectively. The inverse of $T$ is simply $T^{-1} = \{ (y, x) \in \Rl \times \Rk : (x, y) \in T \}$, i.e., $x \in T^{-1}(y)$ if and only if $y \in T(x)$.

By identifying possibly multivalued mappings with their graphs, we can apply notions of set convergence to study convergence of such mappings.
The inner and outer limits of a sequence $(A_n)_n$ of subsets of $\reals^m$ are defined as
\begin{multline*}
	\liminf_{n \to \infty} A_n 
	= \{ x \in \reals^m : \text{there exist $x_n \in A_n$ such that $x_n \to x$ as $n \to \infty$}\} \\
	= \{ x \in \reals^m : \text{for every open $G$ that includes $x$, we have $A_n \cap G \ne \emptyset$ for all but finitely many $n$}\}
\end{multline*}
and
\begin{align*}
	\limsup_{n \to \infty} A_n &= \{ x \in \reals^m : \text{there exists infinite $N \subset \NN$ and $x_n \in A_n$ for $n \in N$ such that $x_n \to x$ as $n \to \infty$}\} \\
	&= \{ x \in \reals^m : \text{for every open $G$ that includes $x$, we have $A_n \cap B \ne \emptyset$ for infinitely many $n$}\}.
\end{align*}
Clearly, the inner limit is contained in the outer limit.
The sequence $A_n \subset \reals^m$ converges in the Painlevé--Kuratowski sense to $A \subset \reals^m$ if $A = \liminf_{n \to \infty} A_n = \limsup_{n \to \infty} A_n$. The inner and outer limits are necessarily closed. If the sets $A_n$ are closed as well, this notion of set convergence is equivalent to convergence in the Fell topology; see for instance \citet[Theorem~4.5]{Rockafellar-Wets}.
Appendix~\ref{app:Fell} provides background and auxiliary results on the Fell topology. 

\subsection{Closed multivalued mappings between Euclidean spaces}
\label{subsect:clmultiEucl}

If $T \subset \reals^k \times \reals^\ell$ is closed, i.e., if $T \in \Fell(\Rkl)$, the corresponding multivalued mapping $T : \reals^k \toto \reals^\ell$ is outer semicontinuous \citep[Theorem~5.7]{Rockafellar-Wets}, which is a property of the values of the map in the neighbourhood of a given point. Here, we extend this property to neighbourhoods of general sets. Let $\ball_m$ and $\oball_m$ denote the closed and open unit balls in $\reals^m$, respectively.

\begin{lemma}[Outer semicontinuity of closed mappings]
	\label{lem:oscK}
	Let $T \subset \reals^k \times \reals^\ell$ be closed and let $A \subset \reals^k$ be non-empty. For every $\eps > 0$ and every $\rho > 0$, there exists an open $U \subset \reals^k$ such that $A \subset U$ and
	\[
	T(U) \cap \rho \ball_\ell
	\subset
	T(A) + \eps \ball_\ell.
	\]
\end{lemma}

For $T \subset \reals^k \times \reals^\ell$ and non-empty $V \subset \reals^k$, let $T \restr V = T \cap (V \times \reals^\ell)$ denote the restriction of $T$ to $V$, i.e., $(T \restr V)(x) = T(x)$ if $x \in V$ and $(T \restr V)(x) = \emptyset$ if $x \in \Rk \setminus V$. If $T \in \Fell(\Rkl)$, we view $T \restr V$ as an element of the Fell space $\Fell(V \times \Rl)$, the space of all closed subsets of $V \times \Rl$ (closed in the trace topology) equipped with the Fell hit-and-miss topology. The following lemma describes Fell neighbourhoods of $T \restr V$ in $\Fell(V \times \Rl)$.

\begin{lemma}[Fell neighbourhoods of mappings]
	\label{cor:Felluniform}
	Let $T \in \Fell(\Rkl)$ and let $V \subset \Rk$ be open and non-empty. For all $\eps, \rho > 0$ and compact $K \subset V$, there exists an open neighbourhood $\mathcal{G}$ of $T \restr V$ in $\Fell(V \times \Rl)$ such that for any $T_1, T_2 \in \Fell(\Rkl)$ with $T_j \restr V \in \mathcal{G}$ for $j = 1, 2$, we have
	\[
		\forall A \subset K, \qquad
		T_1(A) \cap \rho \ball_\ell 
		\subset T_2(A + \eps \ball_k) + \eps \ball_\ell.
	\]
\end{lemma}

\subsection{Graphical convergence relative to an open set}

A sequence of multivalued mappings $T_n : \Rk \toto \Rl$ converges graphically to a multivalued mapping $T$ if, as subsets of $\Rkl$, they converge in the Painlevé--Kuratowski sense. The graphical inner and outer limits of $(T_n)_n$ are the multivalued mappings $\underline{T}$ and $\overline{T}$ whose graphs are the inner and outer limits of those of $T_n$ in the sense above \citep[Definition~5.32]{Rockafellar-Wets}.
If the graphs of $T_n$ are closed in $\Rkl$, graphical convergence is in turn equivalent to convergence in Fell space $\Fell(\Rkl)$.

Proposition~5.33 in \citet{Rockafellar-Wets} states graphical limit formulas at a point, and from there, graphical convergence relative to a set $X \subset \Rk$ is defined on p.~168 of the same reference.
If the set $X$ is open, this yields the following definition.

\begin{definition}[Graphical convergence relative to an open set]
	\label{def:graphconvV}
	Let $T_n, T \in \Fell(\Rkl)$ and let $V \subset \Rk$ be open and non-empty.
	The sequence $(T_n)_n$ is said to converge graphically to $T$ relative to $V$ as $n \to \infty$ if $T(x) = \underline{T}(x) = \overline{T}(x)$ for all $x \in V$ or, equivalently, $T \restr V = \underline{T} \restr V = \overline{T} \restr V$, where $\underline{T}$ and $\overline{T}$ are the inner and outer limits of $(T_n)_n$, respectively. Notation: $T_n \Vto T$ as $n \to \infty$.
\end{definition}

Recall that an accumulation point $x$ of a sequence $(x_n)_n$ in a topological space is a limit of some subsequence $(x_n)_{n \in N}$, for $N \subset \NN$ with $|N| = \infty$. The following result identifies graphical convergence relative to an open set with Fell convergence in a certain subspace.

\begin{proposition}
	\label{prop:Fell:graphconvrelV}		
	Let $T_n, T \in \Fell(\Rkl)$ and let $V \subset \Rk$ be open and non-empty. The following statements are equivalent:
	\begin{itemize}
		\item[(i)]
		For any accumulation point $T'$ of $(T_n)_n$ in $\Fell(\Rkl)$, we have $T' \restr V = T \restr V$, i.e., $T'(x) = T(x)$ for all $x \in V$.
		\item[(ii)]
		$T_n \restr V \to T \restr V$ as $n \to \infty$ in $\Fell(V \times \Rl)$, i.e., for any compact $K \subset V \times \Rd$ such that $T \cap K = \emptyset$ we have $T_n \cap K = \emptyset$ for all large $n$ and for any open $G \subset V \times \Rd$ such that $T \cap G \ne \emptyset$ we have $T_n \cap G \ne \emptyset$ for all large $n$.
		\item[(iii)]
		$T_n \Vto T$ as $n \to \infty$ as in Definition~\ref{def:graphconvV}.
	\end{itemize}
	Any accumulation point $T'$ in (i) can serve as possible limit $T$.
\end{proposition}

\begin{proof}
	This is a corollary to Lemma~\ref{lem:tracesubseq} applied to $\E = \Rkl$ and $\E' = V \times \Rl$.
\end{proof}

Fell space is equipped with its Borel $\sigma$-field, allowing the study of random sets and thus also random multivalued mappings with closed graphs. By Proposition~\ref{prop:Etrace}, the map $\Fell(\Rkl) \to \Fell(V \times \Rl) : T \mapsto T \restr V$ is continuous and thus Borel measurable as soon as $V \subset \reals^k$ is open. The latter property takes care of measurability issues when considering graphical convergence in distribution of random maps.

\begin{proposition}
	\label{prop:graphconvrelVweak}
	Let $(T_n)_n$ be a sequence of random elements in $\Fell(\Rkl)$, let $T \in \Fell(\Rkl)$, and let $V \subset \Rk$ be open and non-empty. The following statements are equivalent:
	\begin{itemize}
		\item[(i)] $T_n \Vto T$ in distribution as $n \to \infty$, i.e., $T_n \restr V \to T \restr V$ weakly in $\Fell(V \times \Rl)$ as $n \to \infty$.
		\item[(ii)] For any random $\hat{T}$ in $\Fell$ that can arise as the limit in distribution of some subsequence of $T_n$, we have $\hat{T} \restr V = T \restr V$ almost surely, i.e., $\prob[\forall x \in V, \, \hat{T}(x) = T(x)] = 1$.
		\item[(iii)] For every compact $K \subset V \times \Rl$ with $T \cap K = \varnothing$ we have $\prob(T_n \cap K = \varnothing) \to 1$ and for every open $G \subset V \times \Rl$ with $T \cap G \ne \varnothing$ we have $\prob(T_n \cap G \ne \varnothing) \to 1$ as $n \to \infty$.
	\end{itemize} 
\end{proposition}

\begin{proof}
	This is a corollary to Proposition~\ref{prop:FellconvrelVweak} applied to $\E = \Rkl$ and $\E' = V \times \Rl$.
\end{proof}

\section{Convergence of monotone mappings between Euclidean spaces}
\label{sec:convmon}

\subsection{Monotone mappings}

The study of the supports of coupling measures in optimal transport involves the following classical notions of monotonicity of multivalued mappings \citep{Rockafellar-Wets, A-A}.

\begin{definition}
	\label{def:mon}
	Let $T \subset \Rdd$. The set $T$ is \emph{monotone} if, for all $(x_1,y_1), (x_2,y_2) \in T$, we have $\inpr{y_2-y_1,x_2-x_1} \ge 0$. The set $T$ is \emph{cyclically monotone} if, for any integer $n \ge 1$ and all $(x_1,y_1),\ldots,(x_n,y_n) \in T$, we have $\sum_{i=1}^n \inpr{x_i,y_i} \ge \sum_{i=1}^n \inpr{x_i,y_{i+1}}$, with $y_{n+1} := y_n$. The set $T$ is \emph{maximal (cyclically) monotone} if it is not contained in a larger subset of $\Rdd$ that is also (cyclically) monotone.
\end{definition}

Since the closure of a (cyclically) monotone set is (cyclically) monotone, a maximal (cyclically) monotone set is necessarily closed. Define the following collections:
\begin{align*}
\Fellm(\Rdd)
&=
\{ T \in \Fell(\Rdd) : \text{$T$ is monotone} \}, \\
\Fellcm(\Rd \times \Rd)
&=
\{ T \in \Fell(\Rdd) : \text{$T$ is cyclically monotone} \}, \\
\Fellmm(\Rd \times \Rd)
&=
\{ T \in \Fell(\Rdd) : \text{$T$ is maximal monotone} \}, \\
\Fellmcm(\Rd \times \Rd)
&=
\{ T \in \Fell(\Rdd) : \text{$T$ is maximal cyclically monotone} \}.
\end{align*}
For simplicity, we write $\Fellm = \Fellm(\Rd \times \Rd)$ and so on. Recall the notation $\Fell_A = \{ F \in \Fell : F \cap A \ne \emptyset \}$ and $\Fell^A = \{ F \in \Fell : F \cap A = \emptyset \}$. The proofs for the results in this section are given in Appendix~\ref{app:convmon}.

%


\begin{lemma}
	\label{lem:m:closed}
	The sets $\Fellm$, $\Fellcm$, $\Fellmm \cup \{\varnothing\}$ and $\Fellmcm \cup \{\varnothing\}$ are all closed in $\Fell(\Rd \times \Rd)$. The complements of $\Fellm$ and $\Fellcm$ are unions of finite intersections of collections $\Fell_G$ for open $G \subset \Rd$.
\end{lemma}

By Lemma~\ref{lem:m:closed}, the collections $\Fellm$, $\Fellcm$, $\Fellmm$ and $\Fellmcm$ are Borel sets of $\Fell(\reals^d \times \reals^d)$.
They are related by
\[
\Fellmcm = \Fellmm \cap \Fellcm 
\quad \text{and} \quad 
\Fellmm \cup \Fellcm \subset \Fellm.
\]
Most inclusions are obvious, but the inclusion $\Fellmcm \subset \Fellmm$ is not \citep[Theorem~12.25]{Rockafellar-Wets}: it says that a cyclically monotone mapping which is not properly contained in another cyclically monotone mapping is not properly contained in another monotone mapping either.

The following proposition gives a sufficient condition under which a maximal cyclically monotone map is determined by its restriction on a set. 

\begin{proposition}[Unique extension]
	\label{prop:unique-extension}
	Let $T \in \Fellmcm$ be such that $\rge T$ is bounded. Let the open set $V \subset \Rd$ satisfy $T(V) \subset \intr(\rge T)$ and $\rge T \subset \cl T(V)$. For any $S \in \Fellmcm$, if $\rge S \subset \cl(\rge T)$, then $S \restr V = T \restr V$ implies $S = T$.
\end{proposition}

\subsection{Convergence of monotone mappings}

Our main theorem involves graphical convergence of maximal cyclically monotone mappings relative to an open set. The next series of results shows how to leverage this property to extract all kinds of other useful convergence properties, up to, under side conditions, uniform convergence on $\Rd$. Many of the results already apply to the larger class of maximal monotone mappings, rather than maximal cyclically monotone ones.

\begin{lemma}[Domain and locally bounded range]
	\label{lem:TUB}
	Let $T \in \Fellmm(\Rdd)$ and let $K \subset V \subset \dom T$, where $K$ is compact and $V$ is open. There exists an open neighbourhood $\mathcal{G}$ of $T \restr V$ in $\Fell(V \times \Rd)$, an open set $U \subset V$ containing $K$ and a bounded set $B \subset \Rd$ such that every $T' \in \Fellmm$ such that $T' \restr V \in \mathcal{G}$ satisfies $U \subset \dom T'$ and $T'(K) \subset B$; in particular, $T'(K)$ is compact. Since the map $S \mapsto S \restr V : \Fell(\Rdd) \to \Fell(V \times \Rd)$ is continuous, this includes every $T' \in \Fellmm$ in some open neighbourhood of $T$ in $\Fell(\Rdd)$.
	As a consequence, there exists an open set $U \subset V$ containing $K$ such that for any sequence $(T_n)_n$ in $\Fellmm$, if $T_n \Vto T$ as $n \to \infty$, then $U \subset \dom T_n$ and $T_n(K) \subset B$ for all but finitely many $n$.
\end{lemma}

The Hausdorff distance $d_H(K, L)$ between two compact subsets $K$ and $L$ of $\Rd$ is the infimum over all $\eps > 0$ such that $K \subset L + \eps \ball$ and $L \subset K + \eps \ball$. If one of $K$ and $L$ is empty while the other is not, the Hausdorff distance is defined to be infinity.

\begin{lemma}[From graphical to local Hausdorff convergence]
	\label{lem:TmmKdeltaB}
	Let $T \in \Fellmm$, let $V \subset \dom T$ be open and let $K \subset V$ be compact, both $V$ and $K$ non-empty. For every $\eps > 0$, there exists $\eta > 0$ such that for all $\delta \in (0, \eta]$ there is an open neighbourhood $\mathcal{G}$ of $T \restr V$ in $\Fell(V \times \Rd)$ such that for all $T' \in \Fellmm$ with $T' \restr V \in \mathcal{G}$ we have $K + \delta \ball \subset \dom T'$ and
	\[
	d_H \bigl( T'(K + \delta \ball), T(K) \bigr) \le \eps.
	\]
	As a consequence, for any sequence $(T_n)_n$ in $\Fellmm$, if $T_n \Vto T$ as $n \to \infty$, then 
	\[ 
	\lim_{\delta \downarrow 0} \limsup_{n \to \infty} 
	d_H \bigl( T_n(K+\delta\ball), T(K) \bigr) = 0. 
	\]
\end{lemma}

If, additionally, the maximal monotone limit map $T$ is single-valued on a compact set $K$ in the interior of its domain, then the previous result on the Hausdorff distance implies uniform convergence on $K$.

\begin{proposition}[Local uniform convergence]
	\label{prop:TmmKsingval}
	Let $T \in \Fellmm$, let $V \subset \dom T$ be open and let $K \subset V$ be compact and non-empty. Assume that $T$ is single-valued on $K$. For every $\eps > 0$ there exists an open neighbourhood $\mathcal{G}$ of $T \restr V$ in $\Fell(V \times \Rd)$ such that for all $T' \in \Fellmm$ with $T' \restr V \in \mathcal{G}$, we have $K \subset \dom T'$ and
	\[
	\sup_{x \in K} \sup_{y \in T'(x)} |y - T(x)| \le \eps.
	\]
	As a consequence, for any sequence $(T_n)_n$ in $\Fellmm$, if $T_n \Vto T$ as $n \to \infty$, then $K \subset \dom T_n$ for all but finitely many $n$ and $\sup_{x \in K} \sup_{y \in T_n(x)} |y - T(x)| \to 0$ as $n \to \infty$.
\end{proposition}

If $T$ is single-valued only on the boundary $\partial K$, then still, the Hausdorff convergence result in Lemma~\ref{lem:TmmKdeltaB} can be strengthened by removal of the inflation factor $\delta$.

\begin{proposition}[From graphical to Hausdorff convergence (bis)]
	\label{prop:graphHausdorff}
	Let $T \in \Fellmm$, let $V \subset \dom T$ be open and let $K \subset V$ be compact and non-empty. Assume $T$ is single-valued on $\partial K$. For every $\eps > 0$ there exists an open neighbourhood $\mathcal{G}$ of $T \restr V$ in $\Fell(V \times \Rd)$ such that for all $T' \in \Fellmm$ with $T' \restr V \in \mathcal{G}$ we have $K \subset \dom T'$ and
	\[
	d_H \bigl( T'(K), T(K) \bigr) \le \eps.
	\]
	As a consequence, for any sequence $(T_n)_n$ in $\Fellmm$, if $T_n \Vto T$, then $d_H(T_n(K), T(K)) \to 0$ as $n \to \infty$.
\end{proposition}

In certain situations, graphical convergence relative to an open set actually implies graphical convergence on the whole of $\Rd$. This, in turn, can then be leveraged to deduce uniform convergence on unbounded sets.

\begin{proposition}[From graphical convergence relative to an open set to convergence everywhere]
	\label{prop:partialcomplete}
	Let $T \in \Fellmcm$ be such that $\rge T$ is bounded. Let the open set $U \subset \Rd$ satisfy $T(U) \subset \intr(\rge T)$ and $\rge T \subset \cl(T(U))$. For any neighbourhood $\mathcal{G}$ of $T$ in $\Fell(\Rdd)$ there exists a neighbourhood $\mathcal{H}$ of $T \restr U$ in $\Fell(U \times \Rd)$ such that for all $T' \in \Fellmcm$ with $\rge T' \subset \cl(\rge T)$ and $T' \restr U \in \mathcal{H}$, we have $T' \in \mathcal{G}$. Hence, if the sequence $(T_n)_n$ in $\Fellmcm$ is such that $\rge T_n \subset \cl(\rge T)$ for all $n$ and $T_n \restr U \to T \restr U$ as $n \to \infty$ in $\Fell(U \times \Rd)$, then actually $T_n \to T$ as $n \to \infty$ in $\Fell(\Rdd)$.
\end{proposition}


The horizon of $E \subset \Rd$ is defined as\footnote{Inspired from \citet[Section~3.A]{Rockafellar-Wets} in the so-called celestial model for the cosmic closure of $\Rd$.}
\[
\hzn E = \left\{ 
u \in \sphere_{d-1} : 
\exists (x_n)_n \in E, |x_n| \to \infty, x_n / |x_n| \to u 
\right\}.
\]
If $E$ is bounded, then $\hzn E$ is empty; if $E$ is unbounded, $\hzn E$ is a closed, non-empty subset of $\sphere_{d-1}$.

The support function of a non-empty bounded set $C \subset \Rd$ is $\sigma_C(u) = \sup_{v \in C} \inpr{u, v}$ for $u \in \Rd$. A compact, convex set $C$ is said to be strictly convex in direction $u \in \sphere_{d-1}$ if the function $v \mapsto \inpr{u, v}$ attains its maximum on $C$ in a single point, i.e., $\argmax_{v \in C} \inpr{u, v}$ is a singleton and thus an exposed point of $C$.

\begin{theorem}[From graphical convergence to uniform convergence on unbounded sets]
	\label{thm:UCT}
	Let $T \in \Fellmm$ have bounded range. Let $E \subset \Rd$ be unbounded only in directions $u \in \sphere_{d-1}$ in which $C = \cl(\rge T)$ is strictly convex, i.e., for every $u \in \hzn E$ the set $\argmax_{c \in C} \inpr{u, c}$ is a singleton. Suppose $T$ is single-valued on $\cl E$. Then for any $\eps > 0$, there exists a neighbourhood $\mathcal{G}$ of $T$ in $\Fell$ such that for any $T' \in \mathcal{G} \cap \Fellmm$ for which $\rge T' \subset C$, we have
	\[
		\sup_{x \in E} \sup_{y \in T'(x)} |y - T(x)| \le \eps.
	\]
\end{theorem}

\section{Coupling measures and their support}
\label{sec:coupling}

The set of probability measures on $\Rd \times \Rd$ with cyclically monotone support is denoted by
\[ 
	\Pcm(\Rdd) = \left\{ \pi \in \Prob(\Rdd) : \spt \pi \in \Fellcm \right\}.
\]
We write $\Pcm = \Pcm(\Rdd)$ if the base space is clear from the context.
Let
\[
	\varPicm(P, Q) = \varPi(P, Q) \cap \Pcm = \left\{ 
		\pi \in \varPi(P, Q) : 
		\spt \pi \in \Fellcm
	\right\}
\]
be the set of couplings with cyclically monotone support. By Theorem~6 in \citet{McCann}, $\varPicm(P, Q)$ is not empty. Moreover, by Corollary~14 in the same article, $\varPicm(P, Q)$ is a singleton as soon as one of $P$ and $Q$ vanishes on all sets of Hausdorff dimension $d-1$. As discussed in the introduction, the interest in  $\varPicm(P, Q)$ stems from its link to optimal transport with squared Euclidean distance as cost.

This section collects some useful properties related to coupling measures whose support satisfies some monotonicity condition. The main result is Proposition~\ref{prop:pinTn} claiming the graphical convergence of maximal cyclically monotone mappings containing the supports of coupling measures relative to the interior of the support of the first margin of the limiting coupling measure. The proof of the results in this section are given in Appendix~\ref{app:coupling}.


\begin{lemma}[Weak closure and compactness properties of sets of coupling measures]
	\label{lem:Picmclosed}
	(a) The sets $\Pcm$ and $\varPicm(P, Q)$, for $P, Q \in \Prob(\Rd)$, are weakly closed in $\Prob(\Rd \times \Rd)$.
	
	(b) The sets $\{ (P, Q, \pi) : \pi \in \varPi(P, Q) \}$ and $\{ (P, Q, \pi) : \pi \in \varPicm(P, Q) \}$ are weakly closed in $\Prob(\Rd) \times \Prob(\Rd) \times \Prob(\Rd \times \Rd)$.
	
	(c) If $K, L \subset \Prob(\Rd)$ are weakly compact, then $\{ (P, Q, \pi) : P \in K, Q \in L, \pi \in \varPi(P, Q)\}$ is weakly compact too.
\end{lemma}

If the support of a coupling measure is (cyclically) monotone, then it can be contained in more than one maximal (cyclically) monotone mappings. Still, certain parts of those mappings are uniquely defined. In the next lemma, maximal monotonicity is enough rather than maximal cyclical monotonicity.

\begin{lemma}[Uniqueness of the optimal mapping on the interior of the support]
	\label{lem:STV}	
	Let $\pi \in \varPi(P, Q)$ for $P, Q \in \Prob(\Rd)$ and let $S, T \in \Fellmm(\Rd \times \Rd)$ satisfy $\spt \pi \subset S \cap T$. Put  $V = \intr(\spt P)$ and $W = \intr(\spt Q)$. Then $V \subset \dom S \cap \dom T$ and $S \restr V = T \restr V$, as well as $W \subset \rge S \cap \rge T$ and $S^{-1} \restr W = T^{-1} \restr W$.
\end{lemma}

Applying the uniqueness property in Lemma~\ref{lem:STV} to the accumulation points of sequences of maximal cyclically monotone mappings allows to deduce their graphical convergence relative an open set.

\begin{proposition}[Graphical convergence of mapst from weak convergence of measures]
	\label{prop:pinTn}
	Let $\pi_n \in \Pcm(\Rdd)$ and $T_n \in \Fellmcm(\Rdd)$ satisfy $\spt \pi_n \subset T_n$ for all $n$.
	Suppose that $\pi_n \wto \pi \in \varPi(P, Q)$ as $n \to \infty$ for some $P, Q \in \Prob(\Rd)$.
	Then $\pi \in \Pcm(\Rdd)$ too and all accumulation points $T$ of $(T_n)_n$ in $\Fell(\Rdd)$ belong to $\Fellmcm(\Rdd)$, contain $\spt \pi$, coincide on $V = \intr(\spt P) \subset \dom T$, and are the graphical limit of $T_n$ relative to $V$, just like any other $T \in \Fellmcm(\Rdd)$ which contains $\spt \pi$.
\end{proposition}

\section{Proofs of main theorems}
\label{sec:proofs}


\begin{proof}[Proof of Theorem~\ref{thm}]

\noindent \textbf{(a)} By Prohorov's theorem, the weak convergences $P_n \wto P$ and $Q_n \wto Q$, and the inclusion $\pi_n \in \varPi(P_n, Q_n)$, the sequence $(\pi_n)_n$ is tight. Any subsequence therefore contains a further subsequence that converges weakly to some $\pi'$. By Lemma~\ref{lem:Picmclosed}(b), we have $\pi' \in \varPi(P, Q)$.
But the latter set is equal to the singleton $\{\pi\}$. It follows that $\pi_n \wto \pi$ as $n \to \infty$. Lemma~\ref{lem:supp-proj} yields $V = \intr(\spt P) \subset \dom T$ while Proposition~\ref{prop:pinTn} yields $T_n \Vto T$ as $n \to \infty$. 

\smallskip \noindent \textbf{(b)} By (a), we have $T_n \Vto T$ as $n \to \infty$. As $T$ is single-valued on the compact subset $K$ of the open subset $V$ of $\dom T$, the statement follows from Lemma~\ref{lem:TUB} and Propositions~\ref{prop:TmmKsingval} and~\ref{prop:graphHausdorff}. The last claim follows by taking $K = \{x\}$ for $x \in W$.

\smallskip \noindent \textbf{(c)} Since $U$ is an open subset of $V$, the convergence of $T_n \restr V$ to $T \restr V$ in $\Fell(V \times \Rd)$ implies the convergence of $T_n \restr U = (T_n \restr V) \restr U$ to $T \restr U = (T \restr V) \restr U$ in $\Fell(U \times \Rd)$; apply Proposition~\ref{prop:Etrace} to $\E = V \times \Rd$ and $\E' = U \times \Rd$, by which the map $\Fell(V \times \Rd) \to \Fell(U \times \Rd) : S \mapsto S \restr U$ is continuous. Next, apply Proposition~\ref{prop:partialcomplete}.

\smallskip \noindent \textbf{(d)} This is a consequence of convergence $T_n \to T$ as $n \to \infty$ in $\Fell(\Rdd)$ established in~(c) together with Theorem~\ref{thm:UCT}.	
%
%
%
%
\end{proof}

\begin{proof}[Proof of Theorem~\ref{thm:random}]
The proof consists of two parts, according to whether $P_n \wto P$ and $Q_n \wto Q$ as $n \to \infty$ almost surely or weakly.

\smallskip \noindent \textbf{Almost sure convergence.} 
The assumptions imply that there exists an event of probability one on which $P_n \wto P$ and $Q_n \wto Q$ as $n \to \infty$ as well as $\pi_n \in \varPicm(P_n, Q_n)$ and $\spt \pi_n \subset T_n$. For (c), we also need that, on the same event, we have $\rge T_n \subset \cl(\rge T)$ for all $n$. On this event, apply Theorem~\ref{thm} to conclude.

We just need to verify that the suprema in~(b) and~(d) are Borel measurable. But this is a consequence of Lemmas~\ref{lem:dHmeas} and~\ref{lem:measdH} and the fact that the maximal (cyclically) monotone mapping $T$ is continuous in a point as soon as it is single-valued in that point.

\smallskip \noindent \textbf{Weak convergence. (a)} We first show that the sequence of laws of the random triples $(P_n, Q_n, \pi_n)$ in $\Prob(\Rd) \times \Prob(\Rd) \times \Prob(\Rdd)$ is uniformly tight. Let $\eps > 0$. Since the laws of the random distributions $P_n$ and $Q_n$ converge in distribution in $\Prob(\Prob(\Rd))$, there exist compact sets $K, L \subset \Prob(\Rd)$ such that $\prob(P_n \in K) \ge 1  - \eps/2$ and $\prob(Q_n \in L) \ge 1 - \eps/2$ for all integer $n$. As $\pi_n \in \varPi(P_n, Q_n)$ for all $n$, Lemma~\ref{lem:Picmclosed}(c) then implies that for all $n$, the probability that the random triple $(P_n, Q_n, \pi_n)$ belongs to the compact set $\{ (P', Q', \pi') : P' \in K, Q' \in L, \pi' \in \varPi(P', Q')\}$ is at least $1-\eps$, as required.

Suppose that along some subsequence, $(P_n, Q_n, \pi_n)$ converges in distribution to the possibly random triple $(P', Q', \pi')$. Then, by assumption and the continuous mapping theorem, we have $P' = P$ and $Q' = P$ with probability one. The laws of $(P_n, Q_n, \pi_n)$ are concentrated on the set $\{(P_0, Q_0, \pi_0) : \pi_0 \in \varPi(P_0, Q_0)\}$, which is weakly closed in $\Prob(\Rd) \times \Prob(\Rd) \times \Prob(\Rdd)$ by Lemma~\ref{lem:Picmclosed}(b). The Portmanteau Theorem yields $\pi' \in \varPi(P, Q)$ with probability one. Now $\varPi(P, Q)$ is equal to the singleton $\{\pi\}$ by assumption, so that actually $(P', Q', \pi') = (P, Q, \pi)$ with probability one. Since this is true for any convergent subsequence of $(P_n, Q_n, \pi_n)$, we obtain that the whole sequence $(P_n, Q_n, \pi_n)$ converges in distribution to the degenerate law at $(P, Q, \pi)$.

Next, recall that Fell space $\Fell(\Rdd)$ is compact. 	Let $V = \intr(\spt P)$. To show that $T_n \Vto T$, we apply Proposition~\ref{prop:graphconvrelVweak}. Let $\hat{T}$ be the weak limit in $\Fell$ of $T_n$ as $n \to \infty$ in some infinite subset $N$ of $\NN$. Since $\pi_n \wto \pi$ in distribution, the limit being deterministic,\footnote{That is, a fixed $\pi \in \Prob(\Rd \times \Rd)$. As a probability measure on $\Rd \times \Rd$, $\pi$ itself is not necessarily degenerate in a single point.} we also have $(\pi_n, T_n) \to (\pi, \hat{T})$ in distribution as $n \to \infty$ in $N$. In view of Lemma~\ref{lem:sptlsc} applied to $\E = \Rd \times \Rd$ and thanks to the Portmanteau theorem, we have $\prob[\spt \pi \subset \hat{T}] = 1$.
By Lemma~\ref{lem:STV}, on an event with probability one, $\hat{T}$ coincides on $V$ with any other $T' \in \Fellmcm$ such that $\spt \pi \subset T'$, in particular with the given $T$. Since this is true for the limit in distribution of any converging subsequence, criterion~(ii) in Proposition~\ref{prop:graphconvrelVweak} is fulfilled. We conclude that $T_n \Vto T$ in distribution.

\smallskip \noindent \textbf{(b)} We prove the last statement, provided $T$ is single-valued on the whole of $K$. Fix $\eps > 0$. By Proposition~\ref{prop:TmmKsingval} there exists an open neighbourhood $\mathcal{G}$ of $T \restr V$ in $\Fell(V \times \Rd)$ such that for $T' \in \Fellmm$ with $T' \restr V \in \mathcal{G}$, we have $K \subset \dom T'$ and $\sup_{x \in K} \sup_{y \in T'(x)} |y - T(x)| \le \eps$. Since $T_n \Vto T$ as $n \to \infty$ in distribution, the Portmanteau Theorem yields that $\liminf_{n \to \infty} \prob(T_n \in \mathcal{G}) = 1$. But as $T_n \in \Fellmcm \subset \Fellmm$ with probability one, then also $\prob[\sup_{x \in K} \sup_{y \in T_n(x)} |y - T(x)| \le \eps] \to 1$ as $n \to \infty$. Since $\eps > 0$ was arbitrary, we conclude that $\sup_{x \in K} \sup_{y \in T_n(x)} |y - T(x)| \to 0$ in distribution as $n \to \infty$.

The proofs of the other statements follow in the same way from Lemma~\ref{lem:TUB} and Proposition~\ref{prop:graphHausdorff}.

\smallskip \noindent \textbf{(c)} As $U$ is an open subset of $V$, Proposition~\ref{prop:Etrace} implies that the restriction map $\Fell(V \times \Rd) \to \Fell(U \times \Rd) : F \mapsto F \cap (U \times \Rd)$ is continuous. Applying this map to $T_n \restr V$ and $T \restr V$ yields $T_n \restr U$ and $T \restr U$ respectively. By (a), the continuous mapping theorem yields $T_n \restr U \to \hat{T} \restr U$ as $n \to \infty$ in distribution in $\Fell(U \times \Rd)$. Let $\mathcal{G}$ be any neighbourhood of $T$ in $\Fell(\Rdd)$. By Proposition~\ref{prop:partialcomplete}, there exists a neighbourhood $\mathcal{H}$ of $T \restr U$ in $\Fell(U \times \Rd)$ such that for all $T' \in \Fellmcm$ with $\rge T' \subset \cl(\rge T)$ and $T' \restr U \in \mathcal{H}$, we have $T' \in \mathcal{G}$. By the Portmanteau Theorem, $\liminf_{n \to \infty} \prob(T_n \restr U \in \mathcal{H}) = 1$. But since also $T_n \in \Fellmcm$ and $\rge T_n \subset \cl(\rge T)$ with probability one, we find $\lim_{n \to \infty} \prob(T_n \in \mathcal{G}) = 1$. As this holds true for any neighbourhood of $T$ in $\Fell(\Rdd)$, we conclude that $T_n \to T$ as $n \in \infty$ in $\Fell(\Rdd)$.

\smallskip \noindent \textbf{(d)} Let $\eps > 0$. By Theorem~\ref{thm:UCT}, there exists a neighbourhood $\mathcal{G}$ of $T$ in $\Fell(\Rdd)$ such that for any $T' \in \mathcal{G} \cap \Fellmm$ for which $\rge T' \subset \cl(\rge T)$, we have $\sup_{x \in E} \sup_{y \in T'(x)} |y-T(x)| \le \eps$. By~(c) and the Portmanteau Theorem, we have $\liminf_{n \to \infty} \prob(T_n \in \mathcal{G}) = 1$. Since also $T_n \subset \Fellmcm \subset \Fellmm$ and $\rge T_n \subset \cl(\rge T)$ with probability one, we find that $\prob(Z_n \le \eps) \to 1$ as $n \to \infty$, where $Z_n = \sup_{x \in E} \sup_{y \in T_n(x)} |y - T(x)|$. But $\eps > 0$ was arbitrary, and thus $Z_n \to 0$ as $n \to \infty$ in distribution.
\end{proof}


\begin{appendix}

\section{Proofs for Section~\ref{sec:mmm}}
\label{app:mmm}

\begin{proof}[Proof of Lemma~\ref{lem:oscK}]
	Put $M = \rho \ball_\ell \setminus [T(A) + \eps \oball_\ell]$, a closed and therefore compact subset of $\rho \ball_\ell$. For every $x \in A$ and $y \in M$, we have $(x, y) \not\in T$ and thus, as $T^c$ is open, $[(x, y) + \eta_{x,y} (\ball_k \times \ball_\ell)] \cap T = \varnothing$ for some $\eta_{x,y} > 0$. For fixed $x \in A$, the open balls $y + \eta_{x,y} \oball_\ell$ with $y \in M$ cover the compact set $M$, so that we can find a finite set $M_x \subset M$ such that the balls $y' + \eta_{x, y'} \oball_\ell$ with $y' \in M_x$ still cover $M$. Let $\eta_{x} = \min_{y \in M_x} \eta_{x,y}$. For every $y \in M$, we can find $y' \in M_x$ such that $y \in y' + \eta_{x, y'} \oball_\ell$ and thus 
	\[ 
	(x + \eta_x \oball_k) \times \{y\} 
	\subset [(x, y) + \eta_{x,y} (\oball_k \times \oball_\ell)] 
	\subset T^c. 
	\]
	As a consequence, $(x + \eta_x \oball_k) \times M \subset T^c$ for all $x \in A$. Define $U = \bigcup_{x \in A} (x + \eta_x \oball_k)$, an open set containing $A$. Then $U \times M \subset T^c$ and thus, by definition of $M$,
	\[ 
	T(U) \cap \rho \ball_\ell 
	\subset \rho \ball_\ell \setminus M 
	\subset T(A) + \eps \oball_\ell. 
	\qedhere
	\]
\end{proof}

\begin{proof}[Proof of Lemma~\ref{cor:Felluniform}]
	By Lemma~\ref{lem:Fellnbd} applied to $\E = V \times \Rl$ and the compact set $K \times \rho \ball_\ell$, there exists an open neighbourhood $\mathcal{G}$ of $T \restr V$ such that, for any $T_1,T_2 \in \Fell(\Rkl)$ satisfying $T_j \restr V \in \mathcal{G}$ for $j \in 1, 2$, we have
	\begin{equation*}
	T_1 \cap (K \times \rho \ball_\ell)
	= (T_1 \restr V) \cap (K \times \rho \ball_\ell) 
	\subset (T_2 \restr V) + \eps(\ball_k \times \ball_\ell)
	\subset T_2 + \eps(\ball_k \times \ball_\ell).
	\end{equation*}
	
	Let $A \subset K$ and let $y \in T_1(A) \cap \rho \ball_\ell$; we will show that $y \in T_2(A + \eps \ball_k) + \eps \ball_\ell$. 
	There exists $x \in A$ such that 
	\[
	(x, y) \in T_1 \cap (A \times \rho \ball_\ell) 
	\subset T_2 + \eps(\ball_k \times \ball_\ell).
	\]
	As a consequence, there exists $(x', y') \in T_2$ such that $|x' - x| \le \eps$ and $|y' - y| \le \eps$. But then, as required,
	\[
	y \in y' + \eps \ball_\ell
	\subset T_2(x') + \eps \ball_\ell 
	\subset T_2(x + \eps \ball_k) + \eps \ball_\ell
	\subset T_2(A + \eps \ball_k) + \eps \ball_\ell.
	\qedhere
	\]	
\end{proof}

\section{Proofs for Section~\ref{sec:convmon}}
\label{app:convmon}

\begin{proof}[Proof of Lemma~\ref{lem:m:closed}]	
	We show first that the complement of $\Fellcm$ is of the stated form and thus open. Suppose that $T \in \Fell(\Rd \times \Rd)$ is not cyclically monotone. Then we can find points $(x_1, y_1), \ldots, (x_n, y_n)$ in $T$ such that $\sum_{i=1}^n \inpr{x_i, y_i} < \sum_{i=1}^n \inpr{x_i, y_{i+1}}$, with $y_{n+1} := y_1$. By continuity of the scalar product, there exists for each $i \in \{1,\ldots,n\}$ a neighbourhood $G_i$ of $(x_i, y_i)$ such that $\sum_{i=1}^n \inpr{\bar{x}_i, \bar{y}_i} < \sum_{i=1}^n \inpr{\bar{x}_i, \bar{y}_{i+1}}$ for every choice of $((\bar{x}_1, \bar{y}_1), \ldots, (\bar{x}_n, \bar{y}_n))$ in $G_1 \times \cdots \times G_n$, with $\bar{y}_{n+1} := \bar{y}_1$. The set 
	\[ 
	\left\{ 
	T' \in \Fell(\Rd \times \Rd) : 
	\forall i = 1, \ldots, n, \ T' \cap G_i \ne \emptyset 
	\right\} 
	= \Fell_{G_1} \cap \cdots \cap \Fell_{G_n}
	\]
	is open in the Fell topology, includes $T$, and is, by construction, disjoint from $\Fellcm$. It follows that the complement of $\Fellcm$ is a union of such intersections and thus open in the Fell topology.
	
	The case $n = 2$ in the previous argument shows that the set $\Fellm$ is closed in Fell space and has a complement of the stated form too.
	
	By Theorem~12.32 and Corollary~12.33 in \citet{Rockafellar-Wets}, the set $\Fellmm \cup \{\emptyset\}$ is closed in Fell space. In the cited Theorem~12.32, the condition that the limit $T$ be non-empty seems missing. Alternatively, the property is stated in Proposition~1.7 in \citet{A-A}.
	
	The empty mapping is cyclically monotone for trivial reasons, while a cyclically monotone subset is maximal monotone if and only if it is maximal cyclically monotone \citep[Theorem~12.25 and the paragraph right before]{Rockafellar-Wets}. It follows that
	\begin{align*}
	\Fellmcm \cup \{\varnothing\} 
	= (\Fellcm \cap \Fellmm) \cup (\Fellcm \cap \{\varnothing\}) 
	= \Fellcm \cap (\Fellmm \cup \{\varnothing\}),
	\end{align*}
	the intersection of two closed subsets in Fell space.
\end{proof}

\begin{proof}[Proof of Proposition~\ref{prop:unique-extension}]
	Since $T^{-1}$ is maximal (cyclically) monotone too, the set $\rge T = \dom T^{-1}$ is nearly convex, that is, there is a convex set $C$ such that $C \subset \rge T \subset \cl C$ \citep[Theorem~12.41]{Rockafellar-Wets}. As $\rge T$ is bounded, we have $\dom T= \Rd$; similarly for any $S \in \Fellmcm$ such that $\rge S \subset \cl(\rge T)$ \citep[Corollary~12.38(a)]{Rockafellar-Wets}.
	
	Let $S \in \Fellmcm$ satisfy $\rge S \subset \cl(\rge T)$ and $S \restr V = T \restr V$. For $y \in T(V)$, there exists $x \in V$ such that $y \in T(x) = S(x)$, so that $x \in T^{-1}(y) \cap S^{-1}(x) \ne \emptyset$. Since $T(V)$ is dense in $\rge T$, it follows by \citet[Corollary~1.5]{A-A} that $T^{-1}(y) = S^{-1}(y)$ for all $y \in \intr(\rge T)$. Moreover, for $y \not\in \cl(\rge T)$ we have $T^{-1}(y) = S^{-1}(y) = \emptyset$.
	
	By Rockafellar's theorem \citep[Theorem~24.8]{Rockafellar}, there exist proper closed convex functions $\psi$ and $\phi$ such that $T^{-1} = \partial \psi$ and $S^{-1} = \partial \phi$. The domains of $\psi$ and $\phi$ contain $\intr(\rge T)$ and are contained in $\cl(\rge T)$. By the previous paragraphs, $\partial \psi$ and $\partial \phi$ coincide on $\intr(\rge T)$. It follows that there exists $c \in \reals$ such that $\phi = \psi + c$ on $\intr(\rge T)$. We may ensure that $c = 0$ by applying a vertical shift, an operation that does not change the subdifferentials. Moreover, $\psi$ and $\phi$ are infinite outside $\cl(\rge T)$. As $\psi$ and $\phi$ are lower semi-continuous, they must coincide everywhere \citep[Corollary~7.3.4]{Rockafellar}; note that the relative interiors of the domains of $\psi$ and $\phi$ coincide and are equal to $\intr(\rge T)$. But then $\psi = \phi$ and thus $T = S$.
\end{proof}

\begin{proof}[Proof of Lemma~\ref{lem:TUB}]
	As $V$ is open, the restriction map $S \mapsto S \restr V : \Fell(\Rdd) \to \Fell(V \times \Rd)$ is continuous by Proposition~\ref{prop:Etrace} applied to $\E = \Rdd$ and $\E' = V \times \Rd$. If $\mathcal{G}$ is an open neighbourhood of $T \restr V$ in $\Fell(V \times \Rd)$, the collection $\{ T' \in \Fell(\Rdd) : T' \restr V \in \mathcal{G} \}$ is thus an open neighbourhood of $T$ in $\Fell(\Rdd)$.
	
	Since $K$ is compact, $V$ is open, and $K$ is contained in $V$, we can find $\eps > 0$ and a finite set $X \subset V$ such that $X + \eps \oball \subset V$, $K \subset \intr(\con X)$, and $K + \eps \ball \subset \con X'$ for every set $X' \subset \Rd$ that intersects every ball $x + \eps \ball$ for $x \in X$. Define
	\[
	\mathcal{G} = \bigcap_{x \in X} \Fell_{G_x}(V \times \Rd) 
	\quad \text{where} \quad
	G_x = (x, T(x)) + \eps \oball \times \eps \oball.
	\]
	Clearly, $\mathcal{G}$ is an open neighbourhood of $T \restr V$ in $\Fell(V \times \Rd)$. Let $T' \in \Fellmm$ be such that $T' \restr V \in \mathcal{G}$.
	For every $x \in X$, choose one point $(x', y') \in T' \cap G_x$; note that $|x' - x| < \eps$ and $|y' - y| < \eps$ for some $y \in T(x)$. Let $X' \subset \dom T'$ be the (finite) collection of points $x'$ thus obtained; similarly, let $Y'$ be the collection of points $y'$ thus obtained. Since $X'$ intersects every ball $x + \eps \ball$, we have $K + \eps \ball \subset \con X'$; in particular, $d(\bar{x}, \reals^d \setminus \con X') \ge \eps$ for every $\bar{x} \in K$. Since $X$ is a finite and thus compact subset of $\intr(\dom T)$, the image $T(X)$ is bounded \cite[Corollary~1.3(3)]{A-A}. As a consequence, there exists $\rho_1$ not depending on $T'$ such that $|y'| \le \rho_1$. By Proposition~2.1(4) in \cite{A-A}, it follows that for every $\bar{x} \in K$ and any $\bar{y} \in T'(\bar{x})$, the norm $|\bar{y}|$ is bounded by a constant that does not depend on $T'$; specifically,
	\[
	|\bar{y}| \le \frac{\max\{ |y'| : y' \in Y' \} \cdot \operatorname{diam}(X')}{d(\bar{x}, \Rd \setminus \con X')}
	\le \frac{\rho_1 \cdot (\operatorname{diam}(X) + 2\eps)}{\eps}.
	\]
	To conclude the proof, set $U = K + \eps \oball$, which is a subset of $\intr(\con X')$ and thus of $\intr(\dom T')$ by \citet[Corollary~1.3(2)]{A-A}, and set $B = \rho \ball$ with $\rho$ equal to the right-hand side of the previous inequality.
\end{proof}

\begin{proof}[Proof of Lemma~\ref{lem:TmmKdeltaB}]
	By Lemma~\ref{lem:oscK}, there exists $\eta > 0$ such that $T(K + 2 \eta \ball) \subset T(K) + (\eps/2) \ball$; note that $T(K + 2 \eta \ball)$ is bounded when is $\eta > 0$ sufficiently small so that $K + 2 \eta \ball \subset \intr(\dom T)$ \cite[Corollary~1.3(3)]{A-A}. Decreasing $\eta > 0$ if needed, we can ensure $\eta \le \eps/2$ and that the compact set $K + 2 \eta \ball$ is contained in $V$.
	
	By Lemma~\ref{lem:TUB} applied to $K + 2 \eta \ball$, there exists an open set $U \subset V$ containing $K + 2 \eta \ball$, a bounded set $B \subset \Rd$ and an open neighbourhood $\mathcal{G}_0$ of $T \restr V$ in $\Fell(V \times \Rd)$ such that every $T' \in \Fellmm$ with $T' \restr V \in \mathcal{G}_0$ satisfies $U \subset \dom T'$ and $T'(K + 2 \eta \ball) \subset B$.
	
	Let $\rho > 0$ be such that $B \subset \rho \ball$ and let $\delta \in (0, \eta]$. By Lemma~\ref{cor:Felluniform} applied to $K + \eta \ball$, there exists an open neighbourhood $\mathcal{G}_1$ of $T \restr V$ in $\Fell(V \times \Rd)$ such that for all $T_1,T_2 \in \Fell(\Rdd)$ with $T_j \restr V \in \mathcal{G}_1$ for $j = 1, 2$ we have, for all $A \subset K + \eta \ball$, the containment $T_1(A) \cap \rho \ball \subset T_2(A + \delta \ball) + \delta \ball$.
	
	Define $\mathcal{G} = \mathcal{G}_0 \cap \mathcal{G}_1$ and let $T' \in \Fellmm$ be such that $T' \restr V \in \mathcal{G}$. From $T' \restr V \in \mathcal{G}_0$, we obtain $K + \delta \ball \subset U \subset \dom T'$ and $T'(K + \delta \ball) \subset B \subset \rho \ball$. From $T' \restr V \in \mathcal{G}_1$ and $\delta \le \eta \le \eps/2$ we deduce, on the one hand,
	\[
	T'(K + \delta \ball) \subset T(K + 2\delta \ball) + \delta \ball \subset T(K) + (\delta + \eps/2) \ball \subset T(K) + \eps \ball,
	\]
	and on the other hand,
	\[
	T(K) \subset T'(K + \delta \ball) + \delta \ball \subset T'(K + \delta \ball) + \eps \ball.
	\]
	It follows that $d_H\bigl(T'(K + \delta \ball), T(K)\bigr) \le \eps$, as required.
\end{proof}

\begin{proof}[Proof of Proposition~\ref{prop:TmmKsingval}]
	For every $x \in K$, Lemma~\ref{lem:TmmKdeltaB} applied to the compact set $\{x\}$ guarantees there exists $\delta_x \in (0, \eps)$ and a neighbourhood $\mathcal{G}_x$ of $T \restr V$ in $\Fell(V \times \Rd)$ such that for any $T' \in \Fellmm$ with $T' \restr V \in \mathcal{G}_x$ we have $x + \delta_x \ball \subset \dom T'$ and
	\[
	d_H \bigl( T'(x + \delta_x \ball), T(x) \bigr) \le \eps/2.
	\]
	By reducing $\delta_x$ if necessary, we may, by Lemma~\ref{lem:oscK},  ensure that 
	\[
	T(x + \delta_x \ball) \subset T(x) + (\eps/2) \ball.
	\]
	Also for this reduced $\delta_x$, we have
	\[
	T'(x + \delta_x \ball) \subset T(x) + (\eps/2) \ball.
	\]
	The compact set $K$ is covered by the union of the open balls $x + \delta_x \oball$ for $x \in K$. Let $Z \subset K$ be a finite set such that $K$ is already covered by the union of the balls $z + \delta_z \ball$ for $z \in Z$. Put $\mathcal{G} = \bigcap_{z \in Z} \mathcal{G}_z$.
	
	Let $T' \in \Fellmm$ be such that $T' \restr V \subset \mathcal{G}$. Then $K \subset \bigcup_{z \in Z} (z + \delta_z \ball) \subset \dom T'$. For any $x \in K$, we can find $z \in Z$ such that $x \in z + \delta_z \ball$ and thus
	\begin{align*}
	T(x) &\in T(z + \delta_z \ball) \subset T(z) + (\eps/2) \ball, \\
	T'(x) &\subset T'(z + \delta_z \ball) \subset T(z) + (\eps/2) \ball.
	\end{align*}
	Since $T(z)$ is a singleton, both $T(x)$ and $T'(x)$ are contained in the same ball with radius $\eps/2$. It follows that the maximal distance between $T(x)$ and a point in $T'(x)$ is not larger than $\eps$.
\end{proof}

\begin{proof}[Proof of Proposition~\ref{prop:graphHausdorff}]
	We apply Lemma~\ref{lem:TmmKdeltaB} twice, to $K$ and to $\partial K$. We obtain, for sufficiently small $\delta > 0$, an open neighbourhood $\mathcal{G}_1$ of $T \restr V$ in $\Fell(V \times \Rd)$ such that for all $T' \in \Fellmm$ with $T' \restr V \in \mathcal{G}_1$, we have $K + \delta \ball \subset \dom T'$ as well as
	\[
	\max\left\{
	d_H \bigl( T'(K + \delta \ball), T(K) \bigr), \;
	d_H \bigl( T'(\partial K + \delta \ball), T(\partial K) \bigr) 
	\right\} 
	\le \eps.
	\]
	
	Further, we apply Proposition~\ref{prop:TmmKsingval} to $\partial K$, on which $T$ is single-valued. We obtain a neighbourhood $\mathcal{G}_2$ of $T \restr V$ in $\Fell(V \times \Rd)$ such that for all $T' \in \Fellmm$ with $T' \restr V \in \mathcal{G}_2$, we have $\partial K \subset \dom T'$ as well as
	\[
	d_H \bigl( T'(\partial K), \; T(\partial K) \bigr) \le \eps.
	\]
	
	Let $\mathcal{G} = \mathcal{G}_1 \cap \mathcal{G}_2$. For $T' \in \Fellmm$ such that $T' \restr V \in \mathcal{G}$, we find, on the one hand,
	\[
	T'(K) \subset T'(K + \delta \ball) \subset T(K) + \eps \ball
	\]
	and, on the other hand, since\footnote{For $y \in (K + \delta \ball) \setminus K$, let $x \in K$ be such that $|x - y| = d(y, K) \le \delta$. Then necessarily $x \in \partial K$, so $y \in \partial K + \delta \ball$.} $K + \delta \ball = K \cup (\partial K + \delta \ball)$,
	\begin{align*}
	T(K) 
	\subset T'(K + \delta \ball) + \eps \ball 
	&= [T'(K) \cup T'(\partial K + \delta \ball)] + \eps \ball \\
	&= [T'(K) + \eps \ball] \cup [T'(\partial K + \delta \ball) + \eps \ball] \\
	&\subset [T'(K) + \eps \ball] \cup [T(\partial K) + 2 \eps \ball] \\
	&\subset [T'(K) + \eps \ball] \cup [T'(\partial K) + 3 \eps \ball] 
	\subset T'(K) + 3 \eps \ball.
	\end{align*}	
	Replace $\eps$ by $\eps/3$ to obtain the stated claim.
\end{proof}

\begin{proof}[Proof of Proposition~\ref{prop:partialcomplete}]
	Suppose the claim is false. Then there exists a neighbourhood $\mathcal{G}$ of $T$ in $\Fell(\Rdd)$ such that, for any neighbourhood $\mathcal{H}$ of $T \restr U$ in $\Fell(U \times \Rd)$, we can find $T' \in \Fellmcm$ with $\rge T' \subset \cl(\rge T)$ and $T' \restr U \in \mathcal{H}$ but still $T' \not\in \mathcal{G}$. This implies we can find a sequence $(T_n)_n$ in $\Fellmcm$ with $\rge T_n \subset \cl(\rge T)$ for any $n$ and $T_n \restr U \to T \restr U$ as $n \to \infty$ but still $T_n \not\in \mathcal{G}$ for any $n$. Since $T(U)$ is not empty, it follows that no subsequence of $T_n \restr U$ converges to the empty set and therefore neither does any subsequence of $T_n$. Let $S$ be an accumulation point of $T_n$ in $\Fell(\Rdd)$. Then $S$ is not the empty mapping, and thus $S \in \Fellmcm$ (Lemma~\ref{lem:m:closed}). Moreover, $\rge S \subset \liminf \cl(\rge T_n) \subset \cl(\rge T)$ (Lemma~\ref{lem:domrgelsc}) and $S \restr U = T \restr U$. By Proposition~\ref{prop:unique-extension}, we conclude that $S = T$. But this is in contradiction to the starting assumption that $T_n$ is not included in the neighbourhood $\mathcal{G}$ of $T$.
\end{proof}

\begin{lemma}[Support functions]
	\label{lem:support}
	Let $C \subset \Rd$ be convex, compact, and non-empty. Let $\sigma_C(u) = \sup_{c \in C} \inpr{u, c}$, for $u \in \Rd$, be the support function of $C$.
	\begin{itemize} 
		\item[(a)] The subdifferential of the convex function $\sigma_C$ is $		\partial \sigma_C(u) = \argmax_{c \in C} \inpr{u, c}$ for $u \in \Rd$. It is continuous in a point as soon as it is single-valued in that point.
		\item[(b)]
		Let the non-empty set $U \subset \Rd$ be compact and such that $\partial \sigma_C$ is single-valued on $U$. For every $\eps > 0$ there exists $\delta > 0$ such that for any $(u, y) \in U \times C$, if $\inpr{u, y} > \sigma_C(u) - \delta$ then $|y - \partial \sigma_C(u)| < \eps$.
	\end{itemize}
\end{lemma}

\begin{proof}
	(a) The identity for $\partial \sigma_C$ can be found for instance in \citet[Corollary~8.25]{Rockafellar-Wets}. Since the multivalued mapping $\partial \sigma_C$ is maximal (cyclically) monotone, it is continuous as soon as it is single-valued \citep[Corollary~1.3(4)]{A-A}. 
	
	(b) 
	%
	%
	The function $f : U \times C \to \reals : (u, y) \mapsto f(u, y) = \sigma_C(u) - \inpr{u, y}$ is continuous. As $\partial \sigma_C$ is single-valued and thus continuous on $U$, the set $K_\eps = \{ (u, y) \in U \times C : |y - \partial \sigma_C(u)| \ge \eps \}$, for $\eps > 0$, is closed and thus compact. The function $f$ is strictly positive on $K_\eps$, so that its infimum on $K_\eps$, say $\delta$, is strictly positive too. It follows that for any $(u, y) \in U \times C$, if $\inpr{u, y} > \sigma_C(u) - \delta$ then $f(u, y) < \delta$ and thus $(u, y) \in (U \times C) \setminus K_\eps$, that is, $|y - \partial \sigma_C(u)| < \eps$.
\end{proof}

\begin{proof}[Proof of Theorem~\ref{thm:UCT}]
	If $E$ is bounded, then the existence of a neighbourhood $\mathcal{G}$ with the required properties follows from Proposition~\ref{prop:TmmKsingval} with $V = \Rd = \dom T$ and $K = \cl E$. So suppose $E$ is unbounded.
	
	Put $U = \hzn E$. By Lemma~\ref{lem:support}, there exists $\delta > 0$ so that for any $(u, y) \in U \times C$, if $\inpr{u,y} > \sigma_C(u) - \delta$ then $|y - \partial \sigma_C(u)| < \eps/2$, where $\partial \sigma_C(u) = \argmax_{c \in C} \inpr{u, c}$.
	
	Put $\eta = \delta / 4$. The open balls in $\Rd$ with radius $\eta$ and centres in $\rge T$ cover $C = \cl(\rge T)$. Let $W$ be a finite subset of $\rge T$ such that the open balls of radius $\eta$ and centres in $W$ still cover $C$. For every $w \in W$, choose $v_w \in T^{-1}(w)$. For $w \in W$, let $B(w)$ be the open ball in $\Rdd$ with centre $(v_w, w)$ and radius $\eta$. The graph of $T$ hits $B(w)$ for every $w \in W$. Any $T' \in \Fell(\Rdd)$ in the open Fell neighbourhood $\mathcal{G}' = \bigcap_{w \in W} \Fell_{B(w)}$ of $T$ hits all those balls too. Put $\sigma = \max \{ \norm{v_w} : w \in W \} + \eta$. 
	
	Put $\gamma = \operatorname{diam}(C) = \sup \{ |c_1-c_2| : c_1, c_2 \in C\}$. If $\gamma = 0$ then $C$ is a singleton and there is nothing to prove. So suppose $\gamma > 0$ and put $\tau = \delta / (4 \gamma)$. There exists $\rho > 4 \gamma \sigma / \delta$ such that any $x \in E \setminus \rho \ball$ satisfies $d(x / |x|, U) < \tau$, that is, $|x / |x| - u | < \tau$ for some $u \in U$ depending on $x$.\footnote{Otherwise, there would exist points $x_n \in E$ with $x_n/|x_n| \to \infty$ such that $d(x_n / |x_n|, U) \ge \tau$ for all $n$. The accumulation points of $(x_n / |x_n|)_n$ would not be included in $U = \hzn E$, a contradiction.}
	
	Take any $T' \in \Fellmm \cap \mathcal{G}'$ such that $\rge T' \subset C$. For each $w \in W$, let $(v_w', c_w')$ be a point in the intersection of the graph of $T'$ with $B(w)$. We have $|v_w'| \le \norm{v_w} + \eta \le \sigma$ for any $w \in W$.
	
	Let $x \in E \setminus \rho \ball$. There exists $u_x \in U$ with $|x/|x| - u_x| < \tau$. For any $y \in T'(x)$ and any $w \in W$, we have, by monotonicity of $T'$ and by the Cauchy--Schwarz inequality,
	\begin{align*}
	\inpr{y - c_w', u_x}
	&= \inpr{y - c_w', x - v_w'} / \norm{x} 
	+ \inpr{y - c_w', v_w'/\norm{x} - x/\norm{x} + u_x} \\
	&\ge 0 - |y - c_w'| \cdot (|v_w'|/{\norm{x}} + \norm{x/\norm{x} - u_x})
	> - \gamma \cdot (\sigma / \rho + \tau)
	> - (\delta/4 + \delta/4) = -\delta/2.
	\end{align*}
	For any $c \in C$, we can find $w \in W$ such that $|c - w| < \eta$ and thus $|c - c_w'| \le |c - w| + |w - c_w'| < 2 \eta = \delta/2$. But then, for any $y \in T'(x)$ and any $c \in C$, with $w \in W$ related to $c$ as in the previous sentence, we have
	\begin{align*}
	\inpr{y - c, u_x} 
	&= \inpr{y - c_w', u_x} + \inpr{c_w' - c, u_x} \\
	&> -\delta/2 - |c_w' - c| \cdot \norm{u_x}
	> -\delta/2 - \delta/2 \cdot 1 = - \delta.
	\end{align*}
	We find that $\inpr{y, u_x} > \inpr{c, u_x} - \delta$ for every $c \in C$ and thus $\inpr{y, u_x} > \sigma_C(u_x) - \delta$. As $y \in T'(x) \subset C$, the choice of $\delta$ implies $|y - \partial \sigma_C(u_x)| < \eps / 2$. As this inequality holds for all $T' \in \Fell \cap \mathcal{G}'$ with $\rge T' \subset C$, it holds in particular for $T$ itself. It follows that $|T(x) - \partial \sigma_C(u_x)| < \eps/2$ too and thus that $|y - T(x)| \le |y - \partial \sigma_C(u_x)| + |\partial \sigma_C(u_x) - T(x)| < \eps$ for any $y \in T'(x)$ and $T' \in \Fellmm \cap \mathcal{G}'$ with $\rge T' \subset C$.
	
	Proposition~\ref{prop:TmmKsingval} with $V = \Rd = \dom T$ and $K = \rho \ball \cap \cl(E)$ provides an open neighbourhood $\mathcal{G}''$ of $T$ in $\Fell(\Rdd)$ such that for any $T' \in \Fellmm \cap \mathcal{G}''$, we have
	\[
	\sup_{x \in E \cap \rho \ball} \sup_{y \in T'(x)} |y - T(x)| < \eps.
	\]

	Finally, put $\mathcal{G} = \mathcal{G}' \cap \mathcal{G}''$. Then, for any $T' \in \Fellmm \cap \mathcal{G}$ with $\rge T' \subset C$ and for any $x \in E$, we have $\sup_{y \in T'(x)} |y - T(x)| < \eps$, whether $|x|$ is bounded by $\rho$ or not. The proof is complete.
\end{proof}

\section{Proofs for Section~\ref{sec:coupling}}
\label{app:coupling}

\begin{proof}[Proof of Lemma~\ref{lem:Picmclosed}]
	(a) As $\varPicm(P, Q) = \varPi(P, Q) \cap \Pcm$, it suffices to show that $\varPi(P, Q)$ and $\Pcm$ are weakly closed.
	
	The set $\varPi(P, Q)$ is closed by the continuous mapping theorem: if $\pi_n \in \varPi(P, Q)$ converges weakly to $\pi \in \Prob(\Rd \times \Rd)$, then $\pi$ belongs to $\varPi(P, Q)$ too, since the projections $(x, y) \mapsto x$ and $(x, y) \mapsto y$ from $\Rd \times \Rd$ onto $\Rd$ are continuous.
	
	Next, we have $\Pcm = \spt^{-1}(\Fellcm)$ where the map $\spt : \Prob(\Rd) \to \Fell(\Rd)$ sends a probability measure to its support. 
	The set $\spt^{-1}(\Fellcm)$ is closed by lower semi-continuity of $\spt$ (Lemma~\ref{lem:sptlsc}) and the fact that the complement of $\Fellcm$ can be written as a union over finite intersections of sets of the form $\Fell_{G}$ for open $G \subset \Rd \times \Rd$ (Lemma~\ref{lem:m:closed}).
	
	(b) If $\pi_n \in \varPi(P_n, Q_n)$ for all $n$ and if $P_n \wto P$, $Q_n \wto Q$ and $\pi_n \wto \pi$, then also $\pi \in \varPi(P, Q)$ by the continuous mapping theorem. It follows that $\{ (P, Q, \pi) : \pi \in \varPi(P, Q) \}$ is weakly closed in $\Prob(\Rd) \times \Prob(\Rd) \times \Prob(\Rd \times \Rd)$. Intersecting the former with the closed set $\Prob(\Rd) \times \Prob(\Rd) \times \Pcm$ yields the set $\{ (P, Q, \pi) : \pi \in \varPicm(P, Q) \}$, which is thus weakly closed as well.
	
	(c) Let $\eps > 0$. By Prohorov's theorem, there exists $\rho > 0$ such that $P(\rho \ball) \ge 1 - \eps/2$ and $Q(\rho \ball) \ge 1 - \eps/2$ for all $P \in K$ and $Q \in L$. For all such $P$ and $Q$ and for all $\pi \in \varPi(P, Q)$, then $\pi(\rho \ball \times \rho \ball) \ge 1 - \eps$. Again by Prohorov's theorem, $M = \bigcup_{P \in K, Q \in L} \varPi(P, Q)$ has weakly compact closure. The set \[ 
	\{ (P, Q, \pi) : P \in K, Q \in L, \pi \in \varPi(P, Q)\} 
	= \{ (P, Q, \pi) : \pi \in \varPi(P, Q)\} 
	\cap \bigl(K \times L \times \Prob(\Rdd)\bigr)
	\]
	is weakly closed by (b) and is contained in the weakly pre-compact set $K \times L \times M$, and is therefore weakly compact.
\end{proof}

For $j \in \{1, 2\}$, let $\proj_j$ denote the projection $(x_1, x_2) \mapsto x_j$ from $\Rd \times \Rd$ into $\Rd$. The next result has appeared as Lemma~A.5 in \citet{devalk+s:2018}.

\begin{lemma}[Support of a margin]
	\label{lem:supp-proj}
	If $P = (\proj_1)_\# \pi$ is the left marginal of $\pi \in \Prob(\Rdd)$, then $\supp P = \cl( \proj_1 (\supp \pi))$.
	As a consequence, if $\supp \pi \subset T$ for some $T \subset \Rdd$, then $\supp P \subset \cl(\dom T)$.
	If, moreover, $T$ is maximal monotone, then $\intr(\supp P) \subset \dom T$.
\end{lemma}

In Lemma~\ref{lem:supp-proj}, it is not true in general that $\intr(\supp P) \subset \proj_1(\spt \pi)$. A counterexample in $d = 1$ is easily constructed: if $\spt \pi = \{ (x, 1/x) : x \in \reals \setminus \{0\}\}$, then $\spt P = \reals$, but $0 \not\in \proj_1(\spt \pi)$.

\begin{proof}[Proof of Lemma~\ref{lem:STV}]
	The fact that $V$ is contained in both $\dom S$ and $\dom T$ follows from Lemma~\ref{lem:supp-proj}.
	Let $A \subset V$ be a non-empty open ball. It is sufficient to show that $S(x) = T(x)$ for all $x \in A$. To do so, we apply the criterion in Corollary~1.5 in \citet{A-A}: it is sufficient to show that $S(x) \cap T(x) \ne \emptyset$ for all $x$ in a dense subset of $A$.
	Let $U \subset A$ be open and non-empty. By the said criterion, we are done if we can find $x \in U$ such that $S(x) \cap T(x) \ne 0$.
	We have $\pi(U \times \Rd) = P(U) > 0$, since otherwise $\spt P$ would be disjoint with $U$, contradicting $U \subset A \subset \intr(\spt P)$. Since $U \times \Rd$ is open, $\spt \pi$ intersects $U \times \Rd$. Moreover, $\spt \pi \cap (U \times \Rd) \subset \spt \pi \subset S \cap T$. For $(x, y) \in \spt \pi \cap (U \times \Rd)$, we thus have $x \in U$ together with $y \in S(x) \cap T(x)$. The latter intersection is thus not empty, as required.
	
	The statements about $W$ follow by switching the roles of $P$ and $Q$,  noting that $\dom S^{-1} = \rge S$.
\end{proof}

\begin{proof}[Proof of Proposition~\ref{prop:pinTn}]
	Since $\Pcm$ is weakly closed (Lemma~\ref{lem:Picmclosed}), necessarily $\pi \in \Pcm$ too. By Lemma~\ref{lem:sptlsc} and Lemma~\ref{lem:lsc}(iii), we have
	\[
	\spt \pi 
	\subset \liminf_{n \to \infty} \spt \pi_n 
	\subset \liminf_{n \to \infty} T_n.
	\]
	Any accumulation point $T$ of $(T_n)_n$ therefore contains the non-empty set $\spt \pi$ (Lemma~\ref{lem:acc}) and is thus a member of $\Fellmcm$ too (Lemma~\ref{lem:m:closed}). 
	Moreover, by Lemma~\ref{lem:STV} and the relation $\Fellmcm \subset \Fellmm$, all accumulation points $T$ coincide on $V$ and satisfy $V \subset \intr(\dom T)$, just like any other $T \in \Fellmcm$ containing $\spt \pi$.
	Proposition~\ref{prop:Fell:graphconvrelV} implies that $T_n$ converges graphically relative to $V$ and that any $T \in \Fellmcm$ containing $\spt \pi$ is a limit.
\end{proof}

\end{appendix}
%
%

\begin{acks}[Acknowledgments]
The author gratefully acknowledges input from Cees de Valk in an early stage of the paper. The author would like to thank participants of the session on Measure Transportation-Based Statistical Inference at the Joint Statistical Meeting (online, August 2021) and of the workshop on Applied Optimal Transport (IMSI, University of Chicago, May 2022) for helpful discussions, suggestions, and encouragements.
\end{acks}
\begin{funding}
The author gratefully acknowledges support by the grant J.0146.19F (\emph{Cr\'edit de recherche}) of the \emph{Fonds de la Recherche Scientifique - FNRS} (Belgium).
\end{funding}



\bibliographystyle{imsart-nameyear} 
\bibliography{biblio}       

\clearpage

\begin{appendix}
	\setcounter{section}{3}
	
	\section{Supplementary material -- Fell space}
	\label{app:Fell}
	
	Throughout, let $\E$ be a locally compact, Hausdorff second countable (LCHS) space. The basic object of study is the collection of closed subsets of $E$, denoted by $\Fell = \Fell(\E) = \{ F \subset \E : \text{$F$ is closed} \}$. Recall the notation $\Fell_A = \Fell_A(\E) = \{ F \in \Fell(\E) : F \cap A \ne \emptyset \}$ and $\Fell^A = \Fell^A(\E) = \{ F \in \Fell(\E) : F \cap A = \emptyset \}$. The Fell topology on $\Fell$ is the one generated by the collections $\Fell_G$ for open $G \subset \E$ and $\Fell^K$ for compact $K \subset \E$. This Supplement collects some useful results on the Fell topology, focusing on approximation properties of Fell neighbourhoods, inner and outer limits, lower and upper semi-continuity, subspaces, and measurability. The material is inspired from \cite{salinetti+w:1981}, \cite{Rockafellar-Wets} and \cite{Molchanov2005}, formulated and extended in a way suitable for use in the main paper. 
	
	\subsection{Fell neighbourhoods}
	
	Let $d$ be a metric on $\E$ generating the given topology on $\E$. For $x \in \E$ and $A \subset E$, let $d(x, A) = \inf \{ d(x, a) : a \in A \}$, with $\inf \emptyset = \infty$ by convention.
	
	\begin{lemma}[Fell neighbourhoods]
		\label{lem:Fellnbd}
		Let $F \in \Fell(\E)$ and let $d$ be a metric on $\E$ generating its topology. For every $\eps > 0$ and for every compact $K \subset \E$, there exists an open neighbourhood $\mathcal{G}$ of $F$ in $\Fell(\E)$ such that for any $F_1, F_2 \in \mathcal{G}$ we have 
		\[ 
		F_1 \cap K \subset \{ x \in \E : d(x, F_2) < \eps \}.
		\]
	\end{lemma}
	
	\begin{proof}
		If $F$ is empty, then we can simply put $\mathcal{G} = \Fell^K$; so suppose $F$ is non-empty. Let $\ball_x(\rho) = \{ y \in \E : d(x, y) \le \rho \}$ and $\oball_x(\rho) = \{ y \in \E : d(x, y) < \rho \}$ denote the closed and open balls with centre $x \in \E$ and radius $\rho > 0$, respectively.
		Let $X \subset K$ be a finite set such that $K$ is covered by the balls $\ball_x(\eps/3)$ with $x \in X$. Put
		\begin{align*}
		X_1 &= \{ x \in X : F \cap \oball_x(2\eps/3) \ne \emptyset\}, &
		X_2 &= \{ x \in X : F \cap \ball_x(\eps/3) = \emptyset \}.
		\end{align*}
		(Note that there is an open ball in the definition of $X_1$ but a closed one in that of $X_2$ and that the radii are different.) Define
		\[
		\mathcal{G} = \bigcap_{x \in X_1} \Fell_{\oball_x(2\eps/3)} \cap \bigcap_{x \in X_2} \Fell^{\ball_x(\eps/3) \cap K}.
		\]
		By construction, $\mathcal{G}$ is open in $\Fell(\E)$ and includes $F$.
		
		Let $F_1, F_2 \in \mathcal{G}$. If $F_1 \cap K = \emptyset$, there is nothing to prove. Otherwise, let $y \in F_1 \cap K$; we will show that $d(y, F_2) \le \eps$. By definition of $X$, there exists $x \in X$ (depending on $y$) such that $y \in \ball_{x}(\eps/3)$. This $x$ cannot belong to $X_2$, since $y$ belongs to $F_1 \cap \ball_{x}(\eps/3) \cap K$ while $F_1$ misses $\ball_{x}(\eps/3) \cap K$ for all $x \in X_2$. By definition of $X_2$, this means that $F$ hits $\ball_x(\eps/3)$ and thus also $\oball_x(2\eps/3)$, so that $x \in X_1$. But then $F_2$ hits $\oball_x(2\eps/3)$ too, in $z$, say. The triangle inequality yields $d(y, F_2) \le d(y, z) \le d(y, x) + d(x, z) < \eps/3 + 2\eps/3 = \eps$, as required.
	\end{proof}
	
	\subsection{Inner and outer limits}
	
	\begin{definition}
		The \emph{inner} and \emph{outer limits} of a sequence $(F_n)_n$ in $\Fell(\E)$ are defined respectively as
		\begin{align*}
		\liminf_{n \to \infty} F_n 
		&= \{ x \in \E : \text{for any open $G \subset \E$ with $x \in G$, we have $F_n \cap G \ne \varnothing$ for all large $n$}\}, \\
		\limsup_{n \to \infty} F_n 
		&= \{ x \in \E : \text{for any open $G \subset \E$ with $x \in G$, we have $F_n \cap G \ne \varnothing$ infinitely often}\}.
		\end{align*}
		Equivalent definitions are that $x \in \liminf_{n \to \infty} F_n$ if and only if there exist points $x_n \in F_n$ such that $x_n \to x$ as $n \to \infty$, while $x \in \limsup_{n \to \infty} F_n$ if and only if there exist an infinite subset $N \subset \NN$ and points $x_n \in F_n$ for $n \in N$ such that $x_n \to x$ as $n \to \infty$ in $N$.
	\end{definition}
	
	
	\begin{lemma}[Inner limit]
		\label{lem:inner}
		Let $\underline{F} = \liminf_{n \to \infty} F_n$ in $\Fell(\E)$.
		For all $A \subset \E$, the following two statements are equivalent:
		\begin{itemize}
			\item[(i)] $A \subset \underline{F}$;
			\item[(ii)] for every open $G \subset \E$ such that $A \cap G \ne \emptyset$, we have $F_n \cap G \ne \emptyset$ for all large $n$.
		\end{itemize}
		In particular, $\underline{F}$ is closed.
	\end{lemma}
	
	\begin{proof}
		\emph{(i) implies (ii).} ---
		Let $A \subset \underline{F}$ and suppose that $A$ hits some open $G \subset \E$; we need to show that $F_n$ hits $G$ for all large $n$.
		Since $A \cap G \ne \emptyset$, there exists $x \in A \cap G$. Then $x \in \underline{F}$ by assumption and also $x \in G$, so that, by definition of the inner limit, $F_n \cap G \ne \emptyset$ for all large $n$.
		
		\emph{(ii) implies (i).} ---
		Let $x \in A$; we need to show that $x \in \underline{F}$. To that end, let $G \subset \E$ be open and such that $x \in G$; we need to show that $F_n$ hits $G$ for all large $n$. But from $x \in A \cap G$ it follows that $A$ hits $G$, which, by (ii), implies precisely what we want. Hence, $x \in \underline{F}$, as required.
		
		\emph{$\underline{F}$ is closed.} ---
		Let $x \in \E$ have the property that every open $G \subset \E$ that contains $x$ hits $\underline{F}$; we need to show that $x \in \underline{F}$. To that end, let $G \subset \E$ be open and contain $x$; we need to show that $F_n \cap G \ne \emptyset$ for all large $n$. But this follows from (ii) applied to $A = \underline{F}$, since $\underline{F} \cap G \ne \emptyset$, as the intersection contains $x$.
	\end{proof}
	
	%
	
	\begin{lemma}[Outer limit]
		\label{lem:outer}
		Let $\overline{F} = \limsup_{n \to \infty} F_n$ in $\Fell(\E)$. Let $A \subset \E$ and consider the following statements:
		\begin{itemize}
			\item[(i)]
			$A \supset \overline{F}$;
			\item[(ii)]
			for every compact $K \subset \E$ such that $A \cap K = \emptyset$, we have $F_n \cap K = \emptyset$ for all large $n$;
			\item[(iii)]
			$\cl(A) \supset \overline{F}$.
		\end{itemize}
		Then we have the chain of implications (i)\,$\implies$\,(ii)\,$\implies$\,(iii), and these become equivalences if $A$ is itself closed.
		Moreover, $\overline{F}$ is closed.
	\end{lemma}
	
	\begin{proof}
		\emph{(i) implies (ii).} ---
		We proceed by contraposition. Let $A \subset \E$ and suppose there exist compact $K \subset \E$ and $N \subset \NN$ with $|N| = \infty$ such that $A \cap K = \emptyset$ and nevertheless $F_n \cap K \ne \emptyset$ for all $n \in N$. We need to show that $\overline{F} \not\subset A$. Let $x_n \in F_n \cap K$ for every $n \in N$. Since $K$ is compact, we can find $x \in K$ and $M \subset N$ with $|M| = \infty$ such that $x_n \to x$ as $n \to \infty$ in $M$. Let $G \subset \E$ be open and contain $x$. Then $x_n \in F_n \cap G$ for infinitely many $n$ and thus $x \in \overline{F}$. At the same time, $x \in K$ and thus $x \not\in A$, since we assumed that $A \cap K = \emptyset$. It follows that $\overline{F}$ is not a subset of $A$, as required.
		
		\emph{(ii) implies (iii).} ---
		Again, we proceed by contraposition. Suppose $\overline{F}$ is not a subset of $\cl(A)$, so that there exists $x \in \overline{F} \setminus \cl(A)$. We need to find a compact subset $K \subset \E$ such that $A \cap K = \emptyset$ while still $F_n \cap K \ne \emptyset$ for infinitely many $n$. To do so, let $G \subset \E$ be an open neighbourhood of $x$ with compact closure, say $K$, such that $K \cap \cl(A) = \emptyset$ and therefore also $K \cap A = \emptyset$; this is possible since $x \not\in \cl(A)$ and $\E$ is LCHS. Since $x \in \overline{F}$, we have $F_n \cap G \ne \emptyset$ and thus $F_n \cap K \ne \emptyset$ infinitely often, as required.
		
		If $A$ is itself closed, then $\cl(A) = A$ and (iii) trivially implies (i), completing the implication circle.
		
		To show that $\overline{F}$ is closed, let $x \in \E$ be such that any open $G \subset \E$ that contains $x$ hits $\overline{F}$; we need to show that $x \in \overline{F}$. To that end, let $G \subset \E$ be open and contain $x$; to show is that $F_n \cap G \ne \emptyset$ for infinitely many $n$. But $\overline{F} \cap G$ being non-empty by the assumption on $x$, we can find $x' \in \overline{F} \cap G$. The definition of $\overline{F}$ applied to $x'$ implies that $F_n$ hits $G$ infinitely often.
	\end{proof}
	
	\begin{lemma}
		\label{lem:Fellinout}
		In $\Fell(\E)$, we have $F_n \to F$ as $n \to \infty$ if and only if $\liminf_{n \to \infty} F_n = F = \limsup_{n \to \infty} F_n$.
	\end{lemma}
	
	\begin{proof}
		This property is well-known when starting from the traditional definitions of the inner and outer limits, see for instance \citet[Theorem~2.2]{salinetti+w:1981}. Its proof from Lemmas~\ref{lem:inner} and~\ref{lem:outer} is a direct consequence of the fact that $F_n \to F$ as $n \to \infty$ if and only if, for every open $G \subset \E$ such that $F \in \Fell_G$ and every compact $K \subset \E$ such that $F \in \Fell^K$, we have $F_n \in \Fell_G$ and $F_n \in \Fell^K$ for all large $n$.
	\end{proof}
	
	Recall that an accumulation point $x$ of a sequence $(x_n)_n$ in a topological space is a limit point of a subsequence of $(x_n)_n$, that is, there exists $N \subset \NN$ with $|N| = \infty$ such that $x_n \to x$ as $n \to \infty$ in $N$.
	
	\begin{lemma}
		\label{lem:acc}
		Let $F$ be an accumulation point of the sequence $(F_n)_n$ in $\Fell(\E)$. Then $\liminf_{n \to \infty} F_n \subset F \subset \limsup_{n \to \infty} F_n$.
	\end{lemma}
	
	\begin{proof}
		Let $N \subset \NN$ be such that $|N| = \infty$ and $F_n \to F$ in $\Fell(\E)$ as $n \to \infty$ in $N$. By Lemma~\ref{lem:Fellinout},
		\[
		\liminf_{n \to \infty, n \in N} F_n 
		= F 
		= \limsup_{n \to \infty, n \in N} F_n.
		\]
		Also, by definition of the inner and outer limits,
		\[
		\liminf_{n \to \infty} F_n 
		\subset 
		\liminf_{n \to \infty, n \in N} F_n 
		\quad \text{and} \quad
		\limsup_{n \to \infty, n \in N} F_n 
		\subset 
		\limsup_{n \to \infty} F_n.
		\]
		Indeed, if $F_n$ hits an open set $G$ for all but finitely many $n \in \NN$, then also $F_n$ hits $G$ for all but finitely many $n \in N$; while if $F_n$ hits $G$ for infinitely many $n \in N$, then $F_n$ hits $G$ for infinitely many $n \in \NN$.
		The stated inclusions follow.
	\end{proof}
	
	\subsection{Lower and upper semi-continuity}
	
	\begin{definition}
		\label{def:semicnt}
		A map from a topological space into $\Fell(\E)$ is \emph{lower semi-continuous} if and only if for every open $G \subset \E$, the inverse image of $\Fell_G$ is open.
		A map from a topological space into $\Fell(\E)$ is \emph{upper semi-continuous} if and only if for every compact $K \subset \E$, the inverse image of  $\Fell^K$ is open.
	\end{definition}
	
	By definition of the Fell topology, a map is continuous if it is both lower and upper semi-continuous, 
	
	\begin{lemma}[Lower semi-continuity]
		\label{lem:lsc}
		Let $(\DD, \rho)$ be a metric space, let $\E$ be a LCHS space, and let $f : \DD \to \Fell(\E)$. The following statements are equivalent:
		\begin{itemize}
			\item[(i)]
			For every open $G \subset \E$, the set $f^{-1}(\Fell_G)$ is open, i.e., $f$ is lower semi-continuous.
			\item[(ii)]
			If the points $x_n, x \in \DD$ and the open set $G \subset \E$ are such that $f(x) \cap G \ne \varnothing$ and $\lim_{n \to \infty} x_n = x$, then also $f(x_n) \cap G \ne \varnothing$ for all large $n$.
			\item[(iii)]
			If $\lim_{n \to \infty} x_n = x$ in $\DD$, then $f(x) \subset \liminf_{n \to \infty} f(x_n)$.
			\item[(iv)]
			The set $\{(x, F) \in \DD \times \Fell : f(x) \subset F \}$ is closed in the product topology.
		\end{itemize}
	\end{lemma}
	
	\begin{proof}
		The equivalence between (i) and (ii) is the sequential characterisation of an open set in a metric space; note that $x_{(n)} \in f^{-1}(\Fell_G)$ if and only if $f(x_{(n)}) \cap G \ne \emptyset$.
		
		The equivalence between (ii) and (iii) is the characterisation of the Fell inner limit in Lemma~\ref{lem:inner}.
		
		Suppose (iii) holds and let $(x_n, F_n)$ in $\DD \times \Fell$ be such that $f(x_n) \subset F_n$ for all $n$ as well as $x_n \to x$ in $\DD$ and $F_n \to F$ in $\Fell$ as $n \to \infty$. Then $f(x) \subset \liminf_{n \to \infty} f(x_n) \subset \liminf_{n \to \infty} F_n = F$, so (iv) holds.
		
		Conversely, suppose (iii) does not hold: then we can find $x_n \to x$ as $n \to \infty$ in $\DD$ such that $f(x)$ is not a subset of $\liminf_{n \to \infty} f(x_n)$.
		By Lemma~\ref{lem:inner}, there exists an open set $G \subset \E$ such that $f(x) \cap G \ne \emptyset$ but still $f(x_n) \cap G = \emptyset$ for infinitely many $n$.
		Writing $F = \E \setminus G$, which is closed, we find that the pairs $(x_n, F)$ satisfy $f(x_n) \subset F$ along a subsequence but that their limit $(x, F)$ satisfies $f(x) \not\subset F$. If follows that (iv) does not hold either.
	\end{proof}
	
	\begin{lemma}
		\label{lem:ulsc}
		Let $\E_1$ and $\E_2$ be two LCHS spaces and let $u : \E_1 \to \E_2$ be continuous. The map $\Fell(\E_1) \to \Fell(\E_2) : F \mapsto \cl(u(F))$ is lower semi-continuous.
	\end{lemma}
	
	\begin{proof}
		Let $G \subset \E_2$ be open. We need to show that the collection $\mathcal{G} := \{F \in \Fell(\E_1) : \cl(u(F)) \cap G \ne \emptyset \}$ is open in $\Fell(\E_1)$. For $F \in \Fell(\E_1)$, it holds that $\cl(u(F))$ hits $G$ if and only if $u(F)$ hits $G$ if and only if $F$ hits $u^{-1}(G)$. It follows that $\mathcal{G} = \Fell_{u^{-1}(G)}(\E_1)$, which is open in $\Fell(\E_1)$, since $u : \E_1 \to \E_2$ is continuous and $G$ is open in $\E_2$, so that $u^{-1}(G)$ is open in $\E_1$.
	\end{proof}
	
	\begin{lemma}[Lower semi-continuity of domains and ranges]
		\label{lem:domrgelsc}
		The maps from $\Fell(\Rd \times \Rd)$ into $\Fell(\Rd)$ given by $T \mapsto \cl(\dom T)$ and $T \mapsto \cl(\rge T)$ are lower semi-continuous, i.e., if $T_n \to T$ as $n \to \infty$ in $\Fell(\Rd \times \Rd)$, then, in $\Fell(\Rd)$,
		\begin{align*}
		\cl(\dom T) &\subset \liminf_{n \to \infty} \cl(\dom T_n), &
		\cl(\rge T) &\subset \liminf_{n \to \infty} \cl(\rge T_n).
		\end{align*}
	\end{lemma}
	
	\begin{proof}
		The lower semi-continuity of the domain and range maps is a corollary to Lemma~\ref{lem:ulsc} applied to the projection mappings $\proj_j : \Rdd \to \Rd : (x_1, x_2) \mapsto x_j$ for $j \in \{1, 2\}$. The statement about the inner limits of sequences follows from Lemma~\ref{lem:lsc}(ii).
	\end{proof}
	
	Let $\Prob(\E)$ denote the set of Borel probability measures on $\E$. Weak convergence is denoted by $\wto$.
	
	\begin{lemma}[Lower semi-continuity of the support map]
		\label{lem:sptlsc}
		If $P_n \wto P$ in $\Prob(\E)$ as $n \to \infty$, then $\spt P \subset \liminf_{n \to \infty} \spt P_n$, that is, the map 
		\[ 
		\spt : \Prob(\E) \to \Fell(\E) : P \mapsto \spt P 
		\]
		is lower semi-continuous. In particular, $\left\{ (P, F) \in \Prob(\E) \times \Fell(\E) : \spt P \subset F \right\}$ is closed in the product topology.
	\end{lemma}
	
	\begin{proof}
		Let $G \subset \E$ be open and suppose that $\spt(P) \cap G \ne \varnothing$. We need to show that $\spt(P_n) \cap G \ne \varnothing$ for all large $n$. But this follows from the portmanteau lemma for weak convergence, since $\liminf_{n \to \infty} P_n(G) \ge P(G)$ and since $\spt(P_{(n)}) \cap G \ne \varnothing$ if and only if $P_{(n)}(G) > 0$.
		The set $\{(P, F) : \spt P \subset F\}$ is closed in view of Lemma~\ref{lem:lsc}(iv).
	\end{proof}
	
	\subsection{Subspaces}
	
	\begin{proposition}
		\label{prop:Etrace}
		Let $\E' \subset \E$ be LCHS too, where $\E'$ is equipped with the trace topology. The map
		\begin{equation}
		\label{eq:Etrace}
		\Fell(\E) \to \Fell(\E') : F \mapsto F \cap \E'
		\end{equation}
		is upper semi-continuous.
		If $\E'$ is open in $\E$, the map is also lower semi-continuous and thus continuous.
	\end{proposition}
	
	\begin{proof}
		If $F$ is closed in $\E$, then $F \cap \E'$ is closed in $\E'$, so that the map in the statement is well-defined.
		
		Let $K \subset \E'$ be compact.
		Then $K$ is compact as a subset of $\E$ too. We need to show that the inverse image of the set $\Fell_K(\E') = \{ F' \in \Fell(\E') : F' \cap K \neq \emptyset \}$ by the map in \eqref{eq:Etrace} is closed in $\Fell(\E)$. 
		As $K \cap \E' = K$, this inverse image is $\{ F \in \Fell(\E) : F \cap K \neq \emptyset \} = \Fell_K(\E)$, which is indeed closed in $\Fell(\E)$. Hence the map in \eqref{eq:Etrace} is upper semi-continuous.
		
		Assume in addition that $\E'$ is open as a subset of $\E$.
		Let $G \subset \E'$ be open.
		Then $G$ is open as a subset of $\E$ too.
		The inverse image of $\Fell_G(\E')$ by the map~\eqref{eq:Etrace} is $\Fell_G(\E)$, which is open in $\Fell(\E)$, as required.
	\end{proof}
	
	\begin{lemma}
		\label{lem:inner-outer:sub}
		Let $\E' \subset \E$ be open. For every sequence $(F_n)_n$ in $\Fell(\E)$, we have
		\begin{align*}
		\left(\liminf_{n \to \infty} F_n\right) \cap \E' 
		&= \liminf_{n \to \infty} \left(F_n \cap \E'\right), &
		\left(\limsup_{n \to \infty} F_n\right) \cap \E' 
		&= \limsup_{n \to \infty} \left(F_n \cap \E'\right).	
		\end{align*}
		The inner and outer limits on the left and right-hand sides are in $\Fell(\E)$ and $\Fell(\E')$, respectively. 
	\end{lemma}
	
	\begin{proof}
		First we show that $\left(\liminf_{n \to \infty} F_n\right) \cap \E'$ is a subset of $\liminf_{n \to \infty} \left(F_n \cap \E'\right)$. To that end, we apply the criterion of Lemma~\ref{lem:inner} with respect to the subspace $\E'$. Let $G \subset \E'$ be open and suppose that $\left(\liminf_{n \to \infty} F_n\right) \cap \E'$ hits $G$; we need to show that $F_n \cap \E'$ hits $G$ for all large $n$. But as $G$ is also an open subset of $\E$, this follows from the same lemma applied to $\E$ and the fact that $\liminf_{n \to \infty} F_n$ hits $G$, so that $F_n$ must hit $G$ for all large $n$.
		
		Next, we show that $\liminf_{n \to \infty} \left(F_n \cap \E'\right)$ is a subset of $\left(\liminf_{n \to \infty} F_n\right) \cap \E'$. Since the former is already a subset of $\E'$ by definition, it is sufficient to show that it is a subset of $\liminf_{n \to \infty} F_n$ too. To that end, we apply Lemma~\ref{lem:inner} in $\Fell(\E)$. Let $G \subseteq \E$ be open and suppose that $\liminf_{n \to \infty} \left(F_n \cap \E'\right)$ hits $G$; we need to show that $F_n$ hits $G$ for all large $n$. The set $\liminf_{n \to \infty} \left(F_n \cap \E'\right)$ hits $G \cap \E'$, which is open in $\E'$. By the same lemma but now applied to $\Fell(\E')$, the sets $F_n \cap \E'$ must hit $G \cap \E'$ for all large $n$. But then $F_n$ hits $G$ for all large $n$. The criterion in Lemma~\ref{lem:inner} on $\Fell(\E)$ is fulfilled and the inclusion is proved.
		
		
		Next, we show that $\left(\limsup_{n \to \infty} F_n\right) \cap \E'$ is contained in $\limsup_{n \to \infty} \left(F_n \cap \E'\right)$.
		Since both sets are subsets of $\E'$, it is equivalent to show that
		\begin{equation}
		\label{eq:outer:sub:aux}
		\limsup_{n \to \infty} F_n \subset
		\left[ \limsup_{n \to \infty} \left( F_n \cap \E' \right) \right]
		\cup \left( \E \setminus \E' \right).
		\end{equation}
		As $\E'$ is open, the right-hand side is closed in $\E$: indeed, $\limsup_{n \to \infty} \left(F_n \cap \E'\right)$ is closed in $\E'$ and thus of the form $F \cap \E'$ for some closed $F \subset \E$, and the set $(F \cap \E') \cup (\E \setminus \E') = F \cup (\E \setminus \E')$ is a union of two closed sets and thus closed itself.
		We can thus apply the implication (ii)\,$\implies$\,(iii) in Lemma~\ref{lem:outer}.
		To that end, let $K \subset \E$ be compact and suppose that the set on right-hand side in \eqref{eq:outer:sub:aux} misses $K$; we need to show that $F_n$ misses $K$ for all large $n$. Since $\E \setminus \E'$ misses $K$, we have that $K$ is a compact subset of $\E'$. Moreover, $K$ is missed by $\limsup_{n \to \infty} \left( F_n \cap \E' \right)$. By the implication (i)\,$\implies$\,(ii) in Lemma~\ref{lem:outer} applied to $\Fell(\E')$, the sequence $F_n \cap \E'$ misses $K$ for all large $n$. As $K$ is contained in $\E'$, it follows that $F_n$ misses $K$ for all large $n$ too, as required.
		
		Conversely, we show that $\limsup_{n \to \infty} \left(F_n \cap \E'\right)$ is contained in $\left(\limsup_{n \to \infty} F_n\right) \cap \E'$. 
		To that end, we apply the implication (ii)\,$\implies$\,(iii) in Lemma~\ref{lem:outer} to $\Fell(\E')$; note that $\left(\limsup_{n \to \infty} F_n\right) \cap \E'$ is closed in $\E'$. Let $K \subseteq \E'$ be compact and suppose that $K$ is missed by $\left(\limsup_{n \to \infty} F_n\right) \cap \E'$; we need to show that $F_n \cap \E'$ misses $K$ for all large $n$. Since $\E' \cap K = K$, the set $K$ is missed by $\limsup_{n \to \infty} F_n$. But $K$ is a compact subset of $\E$ too, and, in view of the implication (i)\,$\implies$\,(ii) in the same lemma applied to $\Fell(\E)$, the sets $F_n$ miss $K$ for all large $n$. But then $F_n \cap \E'$ misses $K$ for all large $n$, as required.
		%
		%
	\end{proof}

	\begin{lemma}
		\label{lem:tracesubseq}
		Let $\E' \subset \E$ be open and let $F_n, F \in \Fell(\E)$. The following statements are equivalent:
		\begin{itemize}
			\item[(i)] Every accumulation point $F'$ of $(F_n)_n$ in $\Fell(\E)$ satisfies $F' \cap \E' = F \cap \E'$.
			\item[(ii)] $F_n \cap \E' \to F \cap \E'$ as $n \to \infty$ in $\Fell(\E')$.
			\item[(iii)] The inner and outer limits $\underline{F} = \liminf_{n\to\infty} F_n$ and $\overline{F} = \limsup_{n\to\infty} F_n$  satisfy $\underline{F} \cap \E' = F \cap \E' = \overline{F} \cap \E'$.
			\item[(iv)] For every compact $K \subset \E'$ such that $F \cap K = \emptyset$ we have $F_n \cap K = \emptyset$ for all large $n$ and for every open $G \subset \E'$ such that $F \cap G \ne \emptyset$ we have $F_n \cap G \ne \emptyset$ for all large $n$.
		\end{itemize}
		Any accumulation point $F'$ of $(F_n)_n$ can serve as possible limit $F$.
	\end{lemma}
	
	\begin{proof}
		\emph{(i) implies (ii).} ---
		We proceed by contraposition: we show that if statement~(ii) does not hold, statement~(i) does not hold either. Suppose that $F_n \cap \E'$ does not converge to $F \cap \E'$. As $\Fell(\E')$ is compact, we can find $X \in \Fell(\E')$ with $X \neq F \cap \E'$ and a subsequence $N \subset \NN$ such that $F_n \cap \E' \to X$ as $n \to \infty$ in $N$ in $\Fell(\E')$. As $\Fell(\E)$ is compact, we can find $F' \in \Fell(\E)$ and $M \subset N$ with $|M| = \infty$ such that $F_n \to F'$ as $n \to \infty$ in $M$. By Proposition~\ref{prop:Etrace}, we have $F_n \cap \E' \to F' \cap \E'$ as $n \to \infty$ in $M$. But then $F' \cap \E' = X \ne F \cap \E'$, so that (i) does not hold.
		
		\emph{(ii) implies (i).} ---
		Let $F' \in \Fell(\E)$ and $N \subset \NN$ with $|N| = \infty$ be such that $F_n \to F'$ in $\Fell(\E)$ as $n \to \infty$ in $N$.
		By Proposition~\ref{prop:Etrace}, we have $F_n \cap \E' \to F' \cap \E'$ in $\Fell(\E')$ as $n \to \infty$ in $N$.
		By~(ii), we must have $F' \cap \E' = F \cap \E'$, as required.
		
		\emph{(ii) and (iii) are equivalent.} ---
		We have $\lim_{n \to \infty} \left(F_n \cap \E'\right) = F \cap \E'$ in $\Fell(\E')$ if and only if the inner and outer limits of $F_n \cap \E'$ in $\Fell(\E')$ are both equal to $F \cap \E'$. The identities in Lemma~\ref{lem:inner-outer:sub} then imply the equivalence of (ii) and (iii).
		
		\emph{(ii) and (iv) are equivalent.} ---
		This is the characterisation of convergence in $\Fell(\E')$. (Here, it is not required that $\E'$ is open, it suffices that it is LCHS in the trace topology.)
	\end{proof}
	
	\begin{proposition}
		\label{prop:FellconvrelVweak}
		Let $\E' \subset \E$ be open.
		Let $\hat{F}_n$ be random elements in $\Fell(\E)$ and let $F \in \Fell(\E)$. The following statements are equivalent:
		\begin{itemize}
			\item[(i)] $\hat{F}_n \cap \E' \to F \cap \E'$ weakly in $\Fell(\E')$ as $n \to \infty$.
			\item[(ii)] For any random $\hat{F}$ in $\Fell(\E)$ that can arise as the weak limit in $\Fell(\E)$ of some subsequence of $\hat{F}_n$, we have $\hat{F} \cap \E' = F \cap \E'$ almost surely.
			\item[(iii)] For every compact $K \subset \E'$ with $F \cap K = \varnothing$ we have $\prob(\hat{F}_n \cap K = \varnothing) \to 1$ and for every open $G \subset \E'$ with $F \cap G \ne \varnothing$ we have $\prob(\hat{F}_n \cap G \ne \varnothing) \to 1$ as $n \to \infty$.
		\end{itemize} 
	\end{proposition}
	
	\begin{proof}
		\emph{(i) and (iii) are equivalent.} ---
		By the Portmanteau Theorem for weak convergence of probability measures in metric spaces, we have $\hat{F}_n \cap \E' \to F \cap \E'$ weakly in $\Fell(\E')$ if and only if, for every open $\mathcal{U} \subset \Fell(\E')$ containing $F$, we have $\Pr(\hat{F}_n \cap \E' \in \mathcal{U}) \to 1$ as $n \to \infty$.	
		Any such open $\mathcal{U} \subset \Fell(\E')$ is a union of finite intersections of sets of the form $\Fell_G(\E')$ with $G \subset \E'$ open in $\E'$ and $\Fell^K(\E')$ with $K \subset \E'$ compact in $\E'$.
		Since $\E'$ is open, $G \subset \E'$ is open in $\E'$ if and only if it is open in $\E$; further, $K$ is compact as a subset of $\E'$ if and only if it is compact as a subset of $\E$, and this without any further conditions on $\E'$.
		The equivalence now follows easily since $\hat{F}_n \cap \E' \in \Fell^K(\E')$ if and only if $\hat{F}_n \cap K = \varnothing$ (remember $K = K \cap \E'$) and $\hat{F}_n \cap \E' \in \Fell_G(\E')$ if and only if $\hat{F}_n \cap G \ne \varnothing$ (remember $G = G \cap \E'$).
		
		\emph{(i) implies (ii).} ---
		Suppose that $\hat{F}_n \to \hat{F}$ weakly in $\Fell(\E)$ as $n \to \infty$ along some subsequence $N$. By Proposition~\ref{prop:Etrace} and the continuous mapping theorem, we have $\hat{F}_n \cap \E' \to \hat{F} \cap \E'$ weakly in $\Fell(\E')$ as $n \to \infty$ in $N$. But by (i), also $\hat{F}_n \cap \E' \to F \cap \E'$ weakly in $\Fell(\E')$ as $n \to \infty$. It follows that the distribution of $\hat{F} \cap \E'$ as a random element in $\Fell(\E')$ is degenerate at $F \cap \E'$.
		
		\emph{(ii) implies (iii).} ---
		We proceed by contraposition.
		Suppose (iii) does not hold.
		Then there exists a set of the form $\mathcal{U} = \Fell^K(\E')$ or $\mathcal{U} = \Fell_G(\E')$ with $K \subset \E'$ compact or $G \subset \E'$ open such that $F \cap \E' \in \mathcal{U}$ but still $\liminf_{n \to \infty} \prob(\hat{F}_n \cap \E' \in \mathcal{U}) < 1$.
		We can therefore find $\eps > 0$ and an infinite subset $N$ of $\NN$ such that $\prob(\hat{F}_n \cap \E' \in \mathcal{U}) \le 1 - \eps$ for all $n \in N$. Let $\hat{F}$ be the weak limit of $\hat{F}_n$ as $n \to \infty$ along some further infinite subset $M$ of $N$; recall that $\Fell(\E)$ is compact, so that, by Prohorov's theorem, such a subsequence exists.
		In view of Proposition~\ref{prop:Etrace} and the continuous mapping theorem, we have $\hat{F}_n \cap \E' \to \hat{F} \cap \E'$ weakly in $\Fell(\E')$ as $n \to \infty$ in $M$.
		By the Portmanteau Theorem, 
		\[ 
		\prob(\hat{F} \cap \E' \in \mathcal{U}) 
		\le \liminf_{n \to \infty, n \in M} 
		\prob(\hat{F}_n \cap \E' \in \mathcal{U}) 
		\le 1 - \eps. 
		\]
		But since $\mathcal{U}$ is an open neighbourhood of $F \cap \E'$, this means that $\hat{F} \cap \E'$ is not equal to $F \cap \E'$ almost surely, so that (ii) does not hold either.
	\end{proof}
	
\subsection{Measurability}	
\label{app:meas}	

When dealing with random multivalued mappings, measurability issues require some attention. Two useful auxiliary results related to the Hausdorff distance are given next. The Borel $\sigma$-field on Fell space is the one generated by the open sets in the Fell topology. 

\begin{lemma}
	\label{lem:dHmeas}
	Let $K, L \subset \Rd$ be non-empty and compact. The map
	\[
		\Fell(\Rdd) \to [0, \infty] : T \mapsto d_H(T(K), L)
	\]
	is Borel measurable.
\end{lemma}	

\begin{proof}
	Let $\lambda > 0$. It is sufficient to show that the set
	\[
		\left\{ T : d_H(T(K), L) \le \lambda \right\}
	\]
	is Borel measurable. Now $d_H(T(K), L) \le \lambda$ if and only if 
	$T(K) \subset L + \lambda \ball$ and $L \subset T(K) + \lambda \ball$. 
	
	On the one hand, $T(K) \subset L + \lambda \ball$ if and only if $T$ misses $K \times (L + \lambda \ball)^c$. Writing the open set $(L + \lambda \ball)^c$ as a countable union of compacta, say $\bigcup_n M_n$ with $M_n \subset \Rd$, we find
	\[
		\left\{ T : T(K) \subset L + \lambda \ball \right\} 
		= \bigcap_n \Fell^{K \times M_n},
	\]
	a countable intersection of open sets in $\Fell$ and thus Borel measurable.
	
	On the other hand, $L \subset T(K) + \lambda \ball$ if $T(K)$ hits $y + \lambda \ball$ for every $y \in L$, i.e., $T$ hits $K \times (y + \lambda \ball)$ for every $y \in L$. But $K \times (y + \lambda \ball)$ is compact and thus is
	\[
		\left\{ T : L \subset T(K) + \lambda \ball \right\}	
		= \bigcap_{y \in L} \Fell_{K \times (y + \lambda \ball)}
	\]
	an intersection of closed sets and therefore closed as well.
\end{proof}

	Recall that a subset of a topological space is called an $F_\sigma$-set if it can be written as a countable union of closed sets. In an LCHS space, open sets are $F_\sigma$-sets, since they can be written as a countable union of compact sets. 
	
	\begin{lemma}
		\label{lem:measdH}
		Let $A \subset \reals^d$ be an $F_\sigma$-set and let $g : A \to \reals^d$ be continuous. The map
		\[
		\Fell(\reals^d \times \reals^d) \to [0, \infty] :
		T \mapsto \sup_{x \in A} d_H(T(x), \{g(x)\})
		\]
		is Borel measurable.
	\end{lemma}
	
	\begin{proof}
		We can write $A = \bigcup_{n \in \NN} F_n$ where $F_n \subset \reals^d$ is closed for every $n \in \NN$. Since
		\[
		\sup_{x \in A} d_H(T(x), \{g(x)\})
		= \sup_{n \in \NN} \sup_{x \in F_n} d_H(T(x), \{g(x)\}) 
		\]
		and since the pointwise supremum of a sequence of Borel measurable functions is Borel measurable, we can without loss of generality assume that $A$ is closed itself, say $A = F \in \Fell(\reals^d)$.
		
		Let $\lambda > 0$. It is sufficient to show that the set
		\[
		\left\{ 
		T :
		\forall x \in F, \, d_H(T(x), \{g(x)\}) \le \lambda
		\right\}
		\]
		is Borel measurable in $\Fell(\reals^d \times \reals^d)$. For fixed $x \in \reals^d$, we have
		\[
		d_H(T(x), \{g(x)\}) \le \lambda \iff
		\left[ T(x) \subset g(x) + \lambda \ball \text{ and } T(x) \cap (g(x) + \lambda \ball) \ne \varnothing \right].
		\]
		For fixed $x \in \reals^d$, the set $g(x) + \lambda \ball$ is compact and thus the set
		\[ \{ T : T(x) \cap (g(x) + \lambda \ball) \ne \varnothing \} 
		= \left\{ T : T \cap \left[\{x\} \times \left(g(x) + \lambda \ball\right)\right] \ne \varnothing \right\} 
		= \Fell_{g(x) + \lambda \ball}
		\]
		is closed in $\Fell$. Hence, the intersection of these sets over all $x \in F$ is closed too. It remains to be shown that the set
		\[
		\{ T : \forall x \in F, \, T(x) \subset g(x) + \lambda \ball \}
		\]
		is Borel measurable. The tube
		\[
		\{ (x, y) \in F \times \reals^d : |y - g(x)| \le \lambda \}
		\]
		is closed in $\reals^d \times \reals^d$, since $F$ is closed and $g$ is continuous.\footnote{If $(x_n,y_n)$ belongs to the tube for every $n$ and if $(x_n, y_n)$ converges to $(x,y)$, then $x_n \to x \in F$ since $F$ is closed and $|y - g(x)| = \lim_{n \to \infty} |y_n - g(x_n)| \le \lambda$ since $g$ is continuous.}
		The complement of the tube in $\reals^d \times \reals^d$, say $G$, is thus open. 
		The multivalued map $T \in \Fell(\reals^d \times \reals^d)$ is such that $T(x) \subset g(x) + \lambda \ball$ for all $x \in F$ if and only if $T$ misses $G \cap (F \times \reals^d)$. Let $K_n$ be a sequence of compact sets in $\reals^d \times \reals^d$ such that $G = \bigcup_{n \in \NN} K_n$. Then 
		$T \in \Fell$ misses $G \cap (F \times \reals^d)$ if and only if $T$ misses each $K_n \cap (F \times \reals^d)$, i.e.,
		\[
		\Fell^{G \cap (F \times \reals^d)} = \bigcap_{n \in \mathbb{N}} \Fell^{K_n \cap (F \times \reals^d)}.
		\]
		But $K_n \cap (F \times \reals^d)$ is compact, being a closed subset of $K_n$. Hence, $\Fell^{G \cap (F \times \reals^d)}$ is a countable intersection of open subsets of $\Fell$ and thus Borel measurable.
		
		All in all, we find
		\[
		\{ T : T(x) \cap (g(x) + \lambda \ball) \ne \varnothing \}
		=
		\bigcap_{x \in F} \Fell_{g(x) + \lambda \ball}
		\cap
		\bigcap_{n \in \mathbb{N}} \Fell^{K_n \cap (F \times \reals^d)},
		\]
		the intersection of a closed set with a countable intersection of open sets, and thus Borel measurable.
	\end{proof}

\end{appendix}


\end{document}